\documentclass[
11pt]
{amsart}

\usepackage[pagebackref]{hyperref}
\usepackage[margin=1.25in, marginpar=0.5in, marginparsep=.1in]{geometry}
\usepackage{hyperref}

\usepackage{cite}
\usepackage{amssymb}
\usepackage{amsmath}
\usepackage{amsthm}
\usepackage{amsfonts}
\usepackage{graphicx}%
\usepackage{comment}
\usepackage{bbm}
\usepackage{xcolor}
\usepackage[toc,page]{appendix}
\usepackage{mathtools}
\usepackage{amsmath}
\usepackage[utf8]{inputenc}

\numberwithin{equation}{section}

\theoremstyle{plain} 
\newtheorem{thm}{Theorem}[section] 
 
\newtheorem{lem}[thm]{Lemma} 
\newtheorem{prop}[thm]{Proposition} 
\newtheorem{notation}[thm]{Notation} 
\newtheorem{rmk}[thm]{Remark} 
 
\newtheorem*{acknowledgement}{Acknowledgement}

\newcommand{\R}{\mathbb{R}}  
\newcommand{\E}{\mathbb{E}} 
\newcommand{\Prob}{\mathbb{P}}

\title[Chung's LIL and small balls for spherical random fields]{Small fluctuations for time-dependent \\spherical random fields}

\author[Carfagnini]{Marco Carfagnini{$^{\dag }$}}
\thanks{\footnotemark {$\dag$} Research was supported in part by an AMS-Simons Travel Grant}
\address{Department of Mathematics\\
University of California, San Diego\\
La Jolla, CA 92093-0112,  U.S.A.}
\email{mcarfagnini@ucsd.edu}

\author[Todino]{Anna Paola Todino{$^{\star}$}}
\thanks{\footnotemark {$\star$} Research was supported in part by ``INdAM -
GNAMPA Project", codice CUP E53C23001670001.\\}
\address{ Dipartimento di Scienze e Innovazione Tecnologica\\
Universit\`a del Piemonte Orientale\\
Italy}
\email{annapaola.todino@uniupo.it}

\keywords{Time-dependent spherical random fields, spherical harmonics, small ball probabilities, Chung's law of the iterated logarithm.}

\subjclass{60G60, 60G17, 60G15, 42C05. }

\begin{document}

	\begin{abstract}
    We consider time-dependent space isotropic and time stationary spherical Gaussian random fields. We establish  Chung's law of the iterated logarithm  and solve the  small probabilities problem. Our results depend on the high-frequency behaviour of the angular power spectrum: the speed of decay of the small ball probability is faster as either the memory parameter or the space-parameter decreases. \end{abstract}

\maketitle

\tableofcontents

\section{Introduction}
Time-dependent spherical random fields have recently aroused lots of interest as mathematical models in many applications. In particular atmospheric and climate sciences, medical imaging, geophysics, astrophysics and cosmology, see for instance \cite{AlegriaPorcu, BergPorcu2017,CaponeraMarinucci, caponera2023multiscale,CAP, Gal,WP} for recent contributions. In this paper we focus on sample path properties of time-dependent spherical random fields, i.e. defined on $\mathbb{S}^2\times \mathbb{R}$. Fluctuations of geometric functionals associated to such fields have been studied recently in \cite{MRV1, MRV2}. In particular we establish  Chung's law of the iterated logarithm and  small deviations.

We say that a real-valued random field $\{ X_t , t \in \R^{n}\}$ satisfies Chung's law of the iterated logarithm (LIL) at $t_{0} \in \mathbb R^{n}$ with rate function $\phi$ if there exists a constant $C$ such that 
\begin{align}
    & \liminf_{r \rightarrow 0} \phi(r) \max_{t\in B_{r} (t_{0})} \vert X_t - X_{t_{0}} \vert =C  \; \; \; a.s. \label{eqn.chung.LIL}
\end{align}
where $B_{r} (t_{0} )$ is an appropriate ball of radius $r>0$ centered at $t_{0}$. When $X_t$ is a Brownian motion and $t_{0} = 0$, it was proven in a famous paper by K.-L. Chung in 1948 \cite{Chung1948} that \eqref{eqn.chung.LIL} holds with $\phi (r) =  \sqrt{\frac{\log |\log r|}{r}}$ and $C= \frac{\pi}{\sqrt{8}}$. 

Chung's LIL has many applications and we refer to the survey \cite{LiShao2001} for more details. In this paper we are mostly interested in connections to small probabilities. The small ball problem (SBP) consists in finding a function $\theta$ and a constant $c$ such that 
\begin{equation}\label{eqn.sd.general}
    \lim_{\varepsilon \rightarrow 0} \theta (\varepsilon) \log \Prob \left( \max_{t\in B_{r} (t_{0})} \vert X_t -X_{t_0} \vert  < \varepsilon \right) =c. 
\end{equation}

The small ball problem and Chung's LIL for stochastic processes have been extensively studied, and the literature on the subject is vast. In \cite{BaldiRoynette1992b} the authors considered the case of a one-dimensional Brownian motion and H\"older norms, in \cite{KuelbsLi1993a} a Brownian sheet in H\"older norms has been considered, \cite{KhoshnevisanShi1998} studied the integrated Brownian motion in the uniform norm,  and \cite{ChenLi2003} the $m$-fold integrated Brownian motion in both the uniform and $L^{2}$-norm, and in \cite{Carfagnini2022} the $n$-step Kolmogorov diffusion. In \cite{Remillard1994} and \cite{CG20} a small deviation problem and Chung's law of iterated logarithm is proven for a class of stochastic integrals and for the hypoelliptic Brownian motion on the Heisenberg group respectively. 

As far as random fields $\{ X_{t} , t \in \R^{n} \}$ $n \geq 1$ on Euclidean spaces  are concerned,  sample path properties have been studied in both the isotropic and anisotropic case by many authors, see for instance \cite{LWX15, Yimin, LX10, MWX13, T, Xia07, WangXiaoBM,Xia09, WX20, LX19,LX23}. In particular, Chung's LIL for Gaussian random fields with stationary increments has been proven in \cite{LX10}. Estimates on small ball probabilities of random fields have been extensively studied as well, see \cite{T}.

In recent years, sample paths properties of random fields defined on manifolds have also been investigated. In particular, in the case of the sphere $\mathbb{S}^2$, differentiability and H\"older continuity properties of  sample paths have been studied in
\cite{HLS18,LS15}; H\"older conditions of local times and exact moduli of non-differentiability for spherical Gaussian fields in \cite{lan2018holder};
a strong local non determinism of an isotropic spherical Gaussian field and the exact  uniform modulus of continuity have been investigated  in \cite{MXL}. These last results have been then generalized to isotropic Gaussian random fields on compact two-point homogeneous space in \cite{LuMaXi}. Finally we refer to \cite{LAN2019} for regularity properties of the solution to a stochastic wave equation given by a fractional Gaussian noise on $\mathbb{S}^2$. 

In the current paper we consider a time-dependent spherical Gaussian random field $\{ T(x,t) , \; (x,t) \in \mathbb S^{2} \times [0,T]\}$ which is stationary in time and isotropic in space. Our main results are collected in  Theorem \ref{thm.Chung.lil} and Theorem \ref{thm.small.deviations}. The former provides the rate function for Chung's LIL, and the latter describes the small ball probability asymptotics. Both the rate function in \eqref{eqn.chung.LIL} and the asymptotics in \eqref{eqn.sd.general} depend on the high frequency behaviour of the angular power spectrum of the field $\{ T(x,t) , \; (x,t) \in \mathbb S^{2} \times [0,T]\}$. The asymptotic behaviour of the spectral density depends on a space-parameter $\alpha$ and a time-parameter $\beta$. The former describes the high-energy behaviour of the field and the latter can be thought as a "memory" parameter.  The small ball probability dacays to zero  like $\exp \left( \frac{-c}{\varepsilon^{\frac{2}{\beta} + \frac{4}{\alpha-2}}} \right)$, for some positive constant $c$. In particular, for a fixed $\alpha$, the smaller the time-memory parameter $\beta$ the faster the speed of the decay. We refer to \eqref{eqn.phi} and \eqref{eqn.psi} for a more explicit rate and asymptotics function.  

Let us now explain the strategy of the proof. Generally, the first step is to show that the LHS of \eqref{eqn.chung.LIL} is measurable with respect to a tail $\sigma$-field and hence constant. Our approach is to use Theorem \ref{thm.small.deviations} to prove that the LHS of \eqref{eqn.chung.LIL} is constant. More precisely, we first show that is almost surely positive and finite. Then by means of stationarity and isotropy of the field we are able to connect \eqref{eqn.chung.LIL} to \eqref{eqn.sd.general}. Lastly, we show that $C$ in \eqref{eqn.chung.LIL} is constant by giving an explicit expression in terms of $c$ in \eqref{eqn.sd.general}. 

The paper is organized as  follows. In Section \ref{sec2} we recall Talagrand's bounds for random fields, time-spherical Gaussian random fields, and then state the main results of this paper, namely Theorem \ref{thm.Chung.lil} and Theorem \ref{thm.small.deviations}. Section 3 contains some preliminary results which are needed to prove Chung's LIL in Section 4 and small probabilities in Section 5. Lastly, in Appendix A and Appendix B we collect proofs of some technical results.

\section{Background and main results}\label{sec2}

\subsection{Upper and lower bounds for Gaussian random fields} Small ball probabilities of random fields have been extensively studied, see \cite{MWX13,LX,Xia09,X97} and references therein. In this sections we recall results from \cite{T} that we will use throughout the paper.  Let $\{ X_{t} , \, t\in \R^{n} \}$ be a real-valued random field, and $d_{X}$ be its canonical metric, that is,
 \begin{equation}\label{eqn.canonical.metric}
     d_{X} (t,s) := \sqrt{\E \left[ ( X_{t} - X_{s} )^{2} \right]}.
 \end{equation}
One of the main ingredients in our proofs is Talagrand's bound on small ball probabilities for Gaussian processes.

\begin{lem} \cite[Lemma 2.1 and Lemma 2.2]{T}\label{lemma.talagrand}
    Let $\{ X_{t} , \, t\in K \}$ be a separable, real-valued, mean zero Gaussian process indexed by a compact set $K$ with canonical metric $d_{X}$. Let $N(K,d_{X} , \varepsilon)$ be the smallest number of $d_{X}$-balls of radius $\varepsilon$ needed to cover the set $K$. Suppose that there is a function $\Psi : (0,\delta ) \rightarrow (0 , \infty)$ such that $N(K,d_{X} , \varepsilon) \leqslant \Psi (\varepsilon)$ for all $\varepsilon \in (0,\delta)$ and there are constants $1< a_{1} \leqslant a_{2}$ such that for all $\varepsilon \in (0,\delta)$
    \begin{align}\label{eqn.condition.talagrand}
        a_{1} \Psi (\varepsilon ) \leqslant \Psi \left(\frac{\varepsilon}{2} \right) \leqslant a_{2} \Psi (\varepsilon).
    \end{align}
    Then, there is a finite constant $k$ depending only on $a_{1}$ and $a_{2}$ such that for all $u\in (0,\delta)$
    \begin{align}
        \mathbb{P} \left( \sup_{s,t \in K} \vert X_{t} - X_{s} \vert \leqslant u \right) \geqslant \exp\left(- k \Psi (u) \right).
    \end{align}

    Moreover, there exists a positive constant $c$ such that for any $u>0$
    \begin{equation}\label{lemma.talagrand.diameter}
        \mathbb{P} \left( \max_{s,t \in K} \vert X_{s} - X_{t} \vert \geqslant c \left( u+ \int_{0}^{D} \sqrt{\log N(K,d_{X} , \varepsilon)} d\varepsilon \right) \right) \leqslant \exp \left(- \frac{u^{2}}{D^{2}} \right),        
    \end{equation}
    where  $D$ is the diameter of $K$ with respect to $d_{X}$.
\end{lem}

\subsection{Time-dependent spherical random fields}

Let $\mathbb{S}^2$ be the unit 2-dimensional sphere and $\Delta_{\mathbb{S}^2}$ be the Laplace-Beltrami operator on $\mathbb{S}^2$ defined as
\begin{equation*}
\dfrac{1}{\sin \theta }\dfrac{\partial }{\partial \theta }\bigg\{\sin \theta 
\dfrac{\partial }{\partial \theta }\bigg\}+\dfrac{1}{\sin ^{2}\theta }\dfrac{%
	\partial ^{2}}{\partial \varphi ^{2}},\mbox{ }0\leq \theta \leq \pi ,\mbox{ }%
0\leq \varphi < 2\pi .
\end{equation*} It is well-known that the spectrum of $\Delta_{\mathbb{S}^2}$ is purely discrete, its eigenvalues are of the form $-\lambda_\ell=-\ell(\ell+1)$ where $\ell\in \mathbb N$, and the eigenspace corresponding to $\lambda_\ell$ is the $(2\ell+1)$-dimensional space $ \mathcal{S}_{\ell} :=  \text{span}\{Y_{\ell m}(\cdot): m=-\ell, \dots, \ell\}$ of degree-$\ell$ spherical harmonics. Note that $\{ \mathcal{S}_{\ell}, \, \ell \in \mathbb N \}$ is an $L^2(\mathbb S^2, dx)-$orthonormal basis, where $dx$ denotes the Lebesgue measure on $\mathbb S^{2}$. In particular, the functions $Y_{\ell m}$  satisfy
$$\Delta_{\mathbb{S}^2}Y_{\ell m}+\lambda_\ell Y_{\ell m}=0,$$ 
and the addition and duplication formulae
\begin{equation}
\sum_{m=-\ell }^{\ell }Y_{\ell m}(x){Y}_{\ell m}(y)=\frac{2\ell +1}{4\pi }%
P_{\ell }(\left\langle x,y\right\rangle )\text{ ,}  \label{addition}
\end{equation}%
\begin{equation}
\int_{\mathbb{S}^{2}}\frac{2\ell +1}{4\pi }P_{\ell }(\left\langle
x,z\right\rangle )\frac{2\ell +1}{4\pi }P_{\ell }(\left\langle
z,y\right\rangle )dz=\frac{2\ell +1}{4\pi }P_{\ell }(\left\langle
x,y\right\rangle )\text{ ,}  \label{duplication}
\end{equation}
for all $x,y \in \mathbb{S}^{2},$ and $P_\ell:[-1,1]\to \mathbb{R}$, denotes the Legendre polynomial of degree $\ell$:
$$P_\ell(t)= \frac{1}{2^{\ell} \ell!} \frac{d^\ell}{dt^\ell}(t^2-1)^\ell \quad \ell=1,2,\dots$$ and $P_\ell(1)=1.$ We refer to \cite[Eq. (3.42) and Sec. 13.1.2]{MP12} for more details.

A space-time real-valued spherical random field $T$ is a collection of random variables $\{ T(x,t) , \, (x,t) \in \mathbb S^{2}\times \R\}$ defined on a common probability space $(\Omega ,\mathcal{F} , \Prob )$ such that $T: \Omega \times \mathbb S^{2} \times \R \longrightarrow \R$ is a $\mathcal{F}\times \mathcal{F} ( \mathbb S^{2} \times \R )$-measurable map, where $\mathcal{F} ( \mathbb S^{2} \times \R )$ denotes the Borel sigma-algebra of $\mathbb S^{2} \times \R$. We are interested in a space-time real-valued spherical random field $T$ that is 

(1) Gaussian, i.e. its finite-dimensional distributions are Gaussian;

(2) centered, i.e. $\E \left[ T(x,t) \right] =0$ for every $(x,t) \in \mathbb S^{2}\times \R$;

(3) isotropic in space and stationary in time, that is, $T(x,t)$ and $T(g.x, \tau + t)$ have the same distribution for every  $ (x,t) \in \mathbb S^{2}\times \R$, $g \in SO(3)$, and  $\tau \in \R$;

(4) mean-square continuous. 

Condition (3) implies that, for every $ (x,t), (y,s) \in \mathbb S^{2}\times \R$
\[
\E \left[ T(x,t) T(y,s) \right] = \Gamma ( \langle x , y \rangle , t-s ),
\]
where $\langle \cdot , \cdot \rangle$ denotes the standard inner product in $\R^{3}$, and  $\Gamma : [-1,1] \times \R \rightarrow \R$ is a positive-semidefinite function. Under the above conditions, it is well known that  \cite[Theorem 3.3]{BergPorcu2017} \cite[Theorem 3]{MaMalyarenko2020} $\Gamma$ can be written as a uniformly convergent series 
\begin{equation}\label{eqn.cov.field}
    \Gamma(\eta , \tau )= \sum_{\ell=0}^{\infty}  \frac{2\ell+1}{4\pi}  C_\ell(\tau ) P_\ell(\eta) , \quad (\eta , \tau) \in [-1 , 1 ] \times \R ,
\end{equation}
where $\{C_{\ell} (\cdot ) \}_{\ell =0}^{\infty}$ is a sequence of continuous positive semi-definite functions on $\R$, and $\{P_{\ell} \}_{\ell =0}^{\infty}$ denotes the Legendre polynomial. Let 
\begin{align*}
    & a_{\ell m} (t) := \int_{\mathbb S^{2}} T(x,t) Y_{\ell m} (x) dx, \quad t\in \R,
\end{align*}
where $Y_{\ell m}$ denotes the $(\ell , m)-$th spherical harmonic.  By our assumptions it follows that $\{a_{\ell m} , \, m= -\ell , \ldots , \ell \}_{ \ell =0}^{\infty}$ is a family of independent, stationary, centered Gaussian processes on $\R$ such that 
\begin{equation}\label{eqn.covariance.alm}
    E[a_{\ell m}(t)a_{\ell' m'}(s)]=C_\ell(t-s) \delta_\ell^{\ell'} \delta_m^{m'}.
\end{equation}
From the spectral representation for $\Gamma$ one can deduce the following Karhunen–Loève expansion 
\begin{equation}\label{eqn.time.dep.random.field}
    T(x,t)=\sum_{\ell=0}^{\infty} \sum_{m=-\ell}^{\ell} a_{\ell m}(t) Y_{\ell m }(x),
\end{equation}
where the series converges in $L^{2} \left( \Omega \times \mathbb S^{2} \times [0,T] \right)$ for any $T>0$.

Throughout the paper we will assume the following. 

\textbf{Assumption I} There exists  $0<\beta <2$ such that the family of functions $\{C_{\ell} (\cdot ) \}_{\ell =0}^{\infty}$ satisfies 
\begin{equation}\label{eqn.covariance.function}
    C_{\ell} (\tau) = C_\ell(0) ( 1- | \tau |^{\beta}  ),
\end{equation}
for all $ | \tau | <1$. Moreover, there exists an $\alpha$ with  $2+\beta < \alpha <4$ such that 
\begin{align}
   & C_{\ell} (0) = G(\ell) \ell^{-\alpha} >0 , \; \; \ell = 1, 2 , \ldots , \label{condition.power.spectrum} ,
    \\
    & C_{0} (0) = c_{1},
\end{align}
and we assume that 
\begin{align*}
    c_{0}^{-1} \leqslant G(\ell ) \leqslant c_{0},
\end{align*}
for some finite constants $c_{0} \geqslant 1$, $c_{1} >0$. 

\begin{rmk}
    The isotropic Gaussian spherical random field (not time-dependent) $\{Z(x),\, x\in \mathbb{S}^2\}$ with zero mean, can be represented by
\begin{eqnarray}\label{field1}
Z(x)&=&\sum_{\ell=1}^{\infty}\sqrt{C_\ell}T_\ell(x)=\sum_{\ell=1}^{\infty} \sqrt{C_\ell} \sum_{m=-\ell}^{\ell} a_{\ell m} Y_{\ell m}(x), \quad x\in \mathbb S^2,\\
Z_\ell(x)&:=&\sum_{m=-\ell}^{\ell} a_{\ell m} Y_{\ell m}(x), \quad x\in \mathbb S^2,
\end{eqnarray}
where the random variables $\{a_{\ell m}, m=-\ell,\dots, \ell\}$ are standard Gaussian and independent, 
the sequence $\{C_\ell\}_{\ell=0,1,\dots}$ is called angular power spectrum of $Z(x)$ and it is non-negative and such that $\sum_{\ell} C_\ell \frac{2\ell+1}{4\pi}=\mathbb{E}Z^2<\infty$ (see \cite{MP12, MP13}). The series  (\ref{field1}) converges in $L^2(\Omega \times \mathbb{S}^2)$, and in   $L^2(\Omega)$ at every fixed $x$.

The Assumption I on the power spectrum translates here simply as the following.

\textbf{Condition (A):} 
The angular power spectrum $C_\ell$ is such that 
\begin{align}
    & C_{\ell} = G(\ell) \ell^{-\alpha} >0 , \; \; \ell = 1, 2 , \ldots , \label{condition.power.spectrum1}
    \\
    & C_{0} (0) = c_{1}, \notag
\end{align}
where $\alpha>2$ is a constant, and we assume that 
\begin{align*}
    c_{0}^{-1} \leqslant G(\ell ) \leqslant c_{0},
\end{align*}
for some finite constant $c_{0} \geqslant 1$, $c_{1} >0$. 

\end{rmk}

We recall that \eqref{condition.power.spectrum1} has been used in \cite{MXL} and it does not require stronger regularity assumptions such as the ones used in  \cite{BKMP09, LM09, DMT24, ST23} among others. Moreover, Condition (A) also cover the Matérn class of covariance functions  \cite{GF}, which is the prototype of covariance models in the applications. The Condition \eqref{condition.power.spectrum1} with $2< \alpha <4$ is suggested by many models for data analysis of the  Cosmic Microwave Background radiation \cite{Planck, CdM}. For the time-dependent case, the restriction  $2+\beta < \alpha <4$ will be needed in  Proposition \ref{thm.upper.bound}.

\begin{rmk}
    We remark that there is a fractal behaviour for $2<\alpha <4$. For these values of $\alpha$ the modulus of continuity decays slower than linearly with respect to the angular distance, and hence making the sample function of $Z(x)$ non-differentiable,  see \cite[Theorem 2]{MXL}.  Many applied fields do satisfy $2<\alpha < 4$. For example, the value of $\alpha$ for Cosmic Microwave Background data is known to be very close to $2$, see \cite{Plank2014}. For this reason, the estimate \eqref{eqn.bound.variance.1} is not believed to hold true for $\alpha >4$, but it is conjectured for the critical case $\alpha =4$, and we refer to \cite{MXL} for more details. 
\end{rmk} 

Let us consider the function $\mu_{\alpha , \beta} : \mathbb S^{2} \times \R \times \mathbb S^{2} \times \R \longrightarrow [0 , \infty )$ given by  
        \begin{equation}\label{eqn.induced.dist.mu}
            \mu^{2}_{\alpha, \beta} ( (x,t) , (y,s) ) := |t-s|^{\beta} + (1- |t-s|^{\beta}) \rho_{\alpha}^{2} (d_{\mathbb S^{2}} (x,y) ),
        \end{equation}        
    where 
    \begin{equation}\label{eqn.rho}
        \rho_{\alpha} (t) := t^{\frac{\alpha-2}{2}}.
    \end{equation}
    In \cite{MXL} it was shown that the function $\rho_{\alpha}$ is equivalent to the canonical metric on $\mathbb S^{2}$ induced by a spherical random field satisfying \eqref{condition.power.spectrum1} with $2< \alpha <4$. This function was introduced in \cite{MXL} to prove the strong non-local determinism for a spherical Gaussian random field $Z$ on the sphere. More precisely, if \eqref{condition.power.spectrum1} holds for some  $2< \alpha <4$, then there exist two positive and finite constants $C$ and $\varepsilon_{0}$ such that for all integers $n\geqslant 1$, and all $x_{0},\ldots , x_{n} \in \mathbb{S}^{2}$ with $\min_{1\leqslant k \leqslant n} d_{\mathbb{S}^{2}} (x_{0} , x_{k}) \leqslant \varepsilon_{0}$ we have that 
    \begin{equation}\label{eqn.bound.variance.1}
        \text{Var} \left( Z(x_{0}) \, | \, Z(x_{1} ), \ldots , Z(x_{n})   \right) \geqslant C \min_{1\leqslant k \leqslant n} \rho_{\alpha} \left(d_{\mathbb{S}^{2}} (x_{0}, x_{k} )\right)^{2},
    \end{equation}
    where $\rho_{\alpha}$ is given by \eqref{eqn.rho}.

    One of our main ingredients is Proposition \ref{thm.non.local.deter}, where we prove a strong non-local determinism similar to  \eqref{eqn.bound.variance.1} for the time-space random field $T(x,t)$. We end this section by showing that the function $\mu_{\alpha , \beta}$ is equivalent to the canonical distance on $\mathbb S^{2} \times \R$ induced by $T(x,t)$.

\begin{prop}\label{prop.bound.dist}
    Let $\{ T(x,t) , \, (x, t) \in \mathbb S^{2} \times \R \}$ be a Gaussian random field satisfying Assumption I. Then  there exist  constants $1\leqslant c_{\alpha} < \infty$ and $0< \varepsilon <1$ such that for all $x,y\in \mathbb S^{2} $ with $d_{\mathbb S^{2} } (x,y) < \varepsilon $ and all $s,t \in \R$ with $| t-s |<1$ one has that 
    \begin{align}
        c_{\alpha}^{-1} \mu^{2}_{\alpha, \beta} ( (x,t) , (y,s) )\leqslant d_{T}^{2} ( (x,t) , (y,s) ) \leqslant c_{\alpha} \mu^{2}_{\alpha, \beta} ( (x,t) , (y,s) ).
    \end{align}
\end{prop}

\begin{proof}
    By  \eqref{eqn.covariance.function} we have that 
    \begin{align*}
        & d^{2}_{T}( (x,t) , (y,s) ) = \E\left[ T(x,t)^{2} \right] + \E\left[ T(y,s)^{2} \right] - 2\E\left[ T(x,t)T(y,s) \right]
    \\
    & = 2 \sum_{\ell =1}^{\infty}  \frac{2\ell+1}{4\pi} C_\ell(0) - 2 \sum_{\ell =1}^{\infty}  \frac{2\ell+1}{4\pi} C_\ell(t-s) P_\ell(\cos \theta)
    \\
    & = \sum_{\ell =1}^{\infty}  \frac{2\ell+1}{2\pi} (C_{\ell} (0)- C_\ell(t-s) P_\ell(\cos \theta))
    \\
    & = \sum_{\ell =1}^{\infty}  \frac{2\ell+1}{2\pi} \left( C_{\ell} (0)- C_\ell(t-s) \right) +  \sum_{\ell =1}^{\infty}  \frac{2\ell+1}{2\pi}  C_\ell(t-s) \left( 1- P_\ell(\cos \theta) \right)
    \\
    & = | t-s|^{\beta}   \sum_{\ell =1}^{\infty}  \frac{2\ell+1}{2\pi} C_{\ell} (0) + (1-| t-s|^{\beta} )  \sum_{\ell =1}^{\infty}  \frac{2\ell+1}{2\pi} C_{\ell} (0) \left( 1- P_\ell(\cos \theta) \right),
    \end{align*}
    and the result then follows by \cite[Lemma 4]{MXL}.
\end{proof}

\subsection{Main Results}
We want to study the asymptotic behaviour of the maximum of a Gaussian random field  $\{ T(x,t) , \, (x,t) \in \mathbb S^{2} \times \R \}$ in a small time-space regime. Thus, we assume  $(x,t) \in \mathbb S^{2} \times [0,T]$, for some fixed $T>0$.  On the product metric space $\mathbb S^{2} \times [0,T]$ one has the natural $q$-product metrics given by 
\begin{align}
    d_{q} ( (x,t) , (y,s) ) := \left( |t-s|^{q} + d_{\mathbb S^{2}} (x,y)^{q} \right)^{\frac{1}{q}} \label{eqn.dist}
\end{align}
for any $1\leqslant q  \leqslant \infty$. These metrics are all equivalent and induce the same topology. It is easy to see that our results are independent of the choice of $q$, and for this reason we fix $q=2$. Let 
\[
B_{r}(x,t) := \left\{ (y,s) \in \mathbb S^{2} \times [0,T] \, : \,  d_{2} ( (x,t) , (y,s) )< r\right\}.
\]

First, let us state an auxiliary result: the strong non-local determinism for space-time real-valued spherical random fields, whose proof can be found in Appendix \ref{sec:appendixA}.

\begin{prop}\label{thm.non.local.deter} 
    Let $\{ T(x,t) , \, (x, t) \in \mathbb S^{2} \times [0,T] \}$ be a Gaussian random field satisfying Assumption I. Then there exist two positive and finite constants $C$ and $\varepsilon_{0}$ such that for all integers $n\geqslant 1$, and all $(x_{0}, t_{0} ),\ldots ,   ( x_{n}, t_{n} ) \in \mathbb{S}^{2} \times [0,T]$ with $\min_{1\leqslant k \leqslant n} d_{2} \left(  (x_{0}, t_{0}) , (x_{k}, t_{k})\right)\geqslant \varepsilon_{0}$, we have that 
    \begin{equation}\label{eqn.bound.variance.1}
        \text{Var} \left( T(x_{0}, t_{0}) \, | \, T(x_{1}, t_{1} ), \ldots , T(x_{n}, t_{n})   \right) \geqslant 
        C \min_{1\leq k\leq n} \rho_{\alpha}^{2} (d_{\mathbb S^{2}} (x,x_k) )).
    \end{equation}
\end{prop}

Let us introduce the following  notation, and state our main results.

\begin{notation}\label{not.max}
    For any $(x,t)\in \mathbb S^{2} \times [0,T]$, and $r>0$ let us set 
    \begin{align*}
       M_{r} (x,t) := \max_{(y,s)\in B_{r} (x,t)} \vert T(y,s) - T(x,t) \vert ,
    \end{align*}
    where $B_{r} (x,t)$ is defined above. Moreover, for any $r>0$ and $\varepsilon>0$  we define the following functions
    \begin{align}
        &\phi (r) = \phi_{\alpha , \beta} (r) := \left( \frac{\log \vert \log r \vert}{r^{3}}\right)^{p}, \label{eqn.phi}
        \\
        & \psi (r,\varepsilon ) = \psi_{\alpha, \beta} (r,\varepsilon) :=  r^{3} \varepsilon^{-\frac{1}{p}}, \label{eqn.psi}
    \end{align}
    where 
    \begin{equation}\label{eqn.p.rate}
        p=p (\alpha , \beta ):= \left( \frac{2}{\beta} + \frac{4}{\alpha -2} \right)^{-1}.
    \end{equation}
\end{notation}

\begin{thm}\label{thm.Chung.lil}
     Let $\{ T(x,t) , \, (x, t) \in \mathbb S^{2} \times [0,T] \}$ be a time-spherical Gaussian random field satisfying Assumption I. Then there exists a finite positive constant $\kappa_{\alpha , \beta}$  such that, for all $(x,t)\in \mathbb S^{2} \times  [0,T]$ 
    \begin{equation}\label{eqn.Chung.lil.sphere}
        \liminf_{r\rightarrow 0} \phi(r) M_{r} (x,t) = \kappa_{\alpha , \beta}, \; \;  \text{a.s.}
    \end{equation}
     where $\phi = \phi_{\alpha , \beta}$ is given by \eqref{eqn.phi}.
\end{thm}

As an application of Theorem \ref{thm.Chung.lil} we solve the small ball problem for the random field $T(x,t)$.

\begin{thm}\label{thm.small.deviations}
    Let $\{ T(x,t) , \, (x, t) \in \mathbb S^{2} \times [0,T] \}$ be a Gaussian random field satisfying Assumption I, and $\kappa_{\alpha , \beta}$ be given by \eqref{eqn.Chung.lil.sphere}. Then, for every $r>0$
    \begin{equation}\label{eqn.small.deviations}
        \lim_{\varepsilon \rightarrow 0} - \varepsilon^{\frac{1}{p}} \log \mathbb{P} \left( M_{r} (x,t) < \varepsilon \right) = r^{3} \kappa_{\alpha ,\beta}^{\frac{1}{p}},
    \end{equation}
    where $p = p (\alpha , \beta )$ is given by \eqref{eqn.p.rate}.
\end{thm}

We remark that the constant $\kappa_{\alpha ,\beta}$ only depends on $\alpha$ and $\beta$ and not on the point $(x,t)\in \mathbb S^{2} \times [0,T]$. Chung's LIL \eqref{eqn.Chung.lil.sphere} describes sample paths properties of the random field $T(x,t)$ which do not depend on the fixed point $(x,t)\in \mathbb S^{2} \times [0,T]$ due to isotropy and stationarity of the field.

\subsection{Isotropic Gaussian spherical random fields (non-time-dependent)}
Chung's LIL for Gaussian isotropic spherical random fields  (non-time-dependent) can be derived from \cite{LX10} and exploiting the strong non-local determinism in \cite{LuMaXi}. However, it can also be seen as a special case of time-dependent spherical Gaussian fields. Moreover, it is possible here to solve also the small ball probabilities for Gaussian random fields as a special case of Theorem \ref{thm.small.deviations}.

More precisely, let us consider the field in \ref{field1}. For any $x\in \mathbb S^{2}$, and $r>0$ let 
    \begin{align*}
       M_{r} (x) := \max_{y\in B_{r} (x)} \vert Z(y) - Z(x) \vert ,
    \end{align*}
    where $B_{r} (x)$ denotes the geodesic ball in $\mathbb S^{2}$ centered in $x\in \mathbb S^{2}$; and for any $r>0$ and $2<\alpha <4$ 
\begin{align*}
      & \phi_{\alpha} (r) :=    \left( \frac{\log \vert \log r \vert}{r^{2}} \right)^{\frac{\alpha-2}{4}} ,  & \rho_{\alpha} (t) : =t^{\frac{\alpha -2}{2}}.
\end{align*}
    It was shown in \cite{MXL} that the function $\rho_{\alpha}$ is equivalent to the canonical metric $d_{T}$ on $\mathbb S^{2}$ induced by the random field $\{Z(x) , x\in \mathbb S^{2} \}$.
  \begin{rmk}
      Note that the induced distance $  \mu^{2}_{\alpha, \beta} ( (x,t) , (y,s) )$ in (\ref{eqn.induced.dist.mu}) 
        reduces to $\rho_{\alpha}^{2} (d_{\mathbb S^{2}} (x,y) ) $   if we set $t=s$.
  \end{rmk}

    Chung's LIL  and  small ball probabilities  are given by the following two results respectively.

\begin{thm}\label{thm.Chung.lil2}
    Let $\{ Z(x) , \, x\in \mathbb{S}^{2} \}$  be an isotropic Gaussian field satisfying 
    Condition (A) for some $2< \alpha <4$.
    Then there exists a finite positive constant $p_\alpha$  such that 
    \begin{equation*}
        \liminf_{r\rightarrow 0} \phi_{\alpha} (r) M_{r} (x) = p_{\alpha}, \; \text{a.s.}
    \end{equation*}
    for all $x\in \mathbb S^{2}$.
\end{thm}

\begin{thm}\label{thm.small.deviations2}
    Let $\{Z(x) , x \in \mathbb{S}^{2} \}$ be an isotropic Gaussian field satisfying 
    Condition (A) for some $2<\alpha<4$
    and $p_{\alpha}$ be given by Theorem \ref{thm.Chung.lil2}. Then, for every $r>0$
    \begin{equation*}\label{eqn.small.deviations2}
        \lim_{\varepsilon \rightarrow 0} - \varepsilon^{\frac{4}{\alpha-2}} \log \mathbb{P} \left( M_{r} (x) < \varepsilon \right) = r^{2} p_{\alpha}^{\frac{4}{\alpha-2}}.
    \end{equation*}
\end{thm}

\begin{rmk}
    As we also mentioned for the time-dependent spherical Gaussian field, a key tool to prove Theorem \ref{thm.Chung.lil2} (and then Theorem \ref{thm.small.deviations2}) is the strong non-local determinism given in (\ref{eqn.bound.variance.1}) and proved in \cite{MXL}. This result holds if $2<\alpha<4$ and it is conjectured to hold even for $\alpha=4$. If we assume its validity for $\alpha=4$, under the assumption that the random field $\{ Z (x) , \, x\in \mathbb{S}^{2} \}$ satisfies Condition (A) with $\alpha=4$ and \eqref{eqn.bound.variance.1} holds with $\alpha =4$, 
  Chung's LIL  and  small probabilities  can be proven even in this case, with the functions

    \begin{align*}
      &\phi_{4} (r) := \sqrt{\frac{\log \vert \log r \vert}{r^{2} \vert \log r \vert}} ,  &  \rho_{4} (t) : = t \sqrt{\vert \ln t \vert } .
\end{align*}
However, we leave the proof of the strong non-local determinism for spherical random field for $\alpha=4$ as a future research.
\end{rmk}

\section{Preliminary results}
We first prove a weaker version of Theorem \ref{thm.small.deviations}, that is, we show that
\begin{align*}
    & \liminf_{\varepsilon \rightarrow 0} - \varepsilon^{\frac{1}{p}} \log \mathbb{P} \left(M_{r} (x,t)  < \varepsilon \right) >0,
    \\
    & 
    \limsup_{\varepsilon \rightarrow 0} - \varepsilon^{\frac{1}{p}} \log \mathbb{P} \left( M_{r} (x,t)  < \varepsilon \right) < \infty.
\end{align*}
    This is the content of Proposition \ref{thm.small.ball.estimate.upper.bound} and  Proposition \ref{thm.small.ball.estimate.lower.bound}  below.

\begin{prop}\label{thm.small.ball.estimate.upper.bound}
    Let $\{ T(x,t) , \, (x, t) \in \mathbb S^{2} \times [0,T] \}$ be a time-spherical Gaussian random field satisfying Assumption I. Then there exists a positive finite constant $A_{1} = A_{1}( \alpha , \beta)$ only depending on $\alpha$ and $\beta$  such that for all $\varepsilon >0$ and $r>\varepsilon^{\frac{1}{3p}}$ we have that
    \begin{align}\label{eqn.small.ball.upper.bound}
    \mathbb{P} \left(M_{r} (x,t) < \varepsilon \right) \leqslant \exp \left( -A_{1} \psi(r, \varepsilon) \right)    ,
    \end{align}
    where $\psi= \psi_{\alpha, \beta}$ is given by \eqref{eqn.psi}.
\end{prop}

\begin{proof}
    The proof is based on the strong non-local determinism for the field $T(x,t)$ Propositon \ref{thm.non.local.deter}, and on a conditioning argument similar to \cite[Proposition 4.2]{LX}. Let us recall that 
    \[
    B_{r} (x,t):= \left\{ (y,s)\in \mathbb S^{2} \times [0,T] \, : \, d_{\mathbb S^{2}} (x,y)^{2} + | t-s|^{2} < r^{2}  \right\}.
    \]
    Let us consider the sets
    \begin{align*}
        & B_{\frac{r}{\sqrt{2}} } (x) := \left\{ y\in \mathbb S^{2} \, : \,  d_{\mathbb S^{2}} (x,y) < \frac{r}{\sqrt{2}}\right\}, & I_{\frac{r}{\sqrt{2}}} (t) := \left[ 0, t + \frac{r}{\sqrt{2}} \right),
    \end{align*}
    and note that $B_{\frac{r}{\sqrt{2}} } (x) \times  I_{\frac{r}{\sqrt{2}}} (t) \subset  B_{r} (x,t)$, so that 
    \begin{align*}
        \Prob \left( M_{r} (x,t) < \varepsilon \right) \leqslant \Prob \left( \max_{y \in B_{\frac{r}{\sqrt{2}} } (x)}  \max_{s \in I_{\frac{r}{\sqrt{2}}} (t) }  \left| T(y,s) - T(x,t) \right| <\varepsilon  \right).
    \end{align*}
    
    Let $F: A_{r} \rightarrow B_{\frac{r}{\sqrt{2}} } (x)$  be a local diffeomorphism, where $A_{r} \subset \mathbb{R}^{2}$ is a rectangle and $F(0,0) = x$, and for $r$ small enough such that  $[0,r] \times [0,r] \subset A_{r}$. We can choose $F$ in such a way that $d_{\mathbb{S}^{2}} (x_{i} , x_{j} ) \geqslant d_{\mathbb{R}^{2}} \left( (\theta_{i}, \varphi_{i}) , (\theta_{j}, \varphi_{j})  \right)$.

    Let $G: [0,r] \rightarrow [0, t+ r\slash \sqrt{2} ]$ be the bijection given by $G(\eta ) = (t \slash r + 1\slash \sqrt{2} ) \eta$ so that $G(0) =0$ and $G(r) = t+ r \slash \sqrt{2}$. We then have a local diffeomorphism 
    \[
    H : [0,r]\times [0,r]\times [0,r] \longrightarrow B_{\frac{r}{\sqrt{2}} } (x) \times  I_{\frac{r}{\sqrt{2}}} (t) .
    \]
    Let us divide $ [0,r]\times [0,r]$ into sub-rectangles of side-length
    \[
        \frac{r}{\bigg\lfloor \frac{r}{\varepsilon^{\frac{1}{2p}}} \bigg \rfloor +1} ,
    \]
    and $[0,r]$ into sub-rectangles of side-length $r \slash ( \lfloor r \rfloor +1)$. We then get a division of $[0,r]\times [0,r]\times [0,r]$ into $N$ sub-rectangles, where 
    \[
    N =  \left( \bigg\lfloor \frac{r}{\varepsilon^{\frac{1}{2p}}} \bigg \rfloor +1 \right)^{2} ( \lfloor r \rfloor +1) \geqslant  \frac{r^{3}}{\varepsilon^{\frac{1}{p}}}.
    \]
    Note that $N > 1$ since $r>\varepsilon^{\frac{1}{3p}}$. Now let $(\theta_{i} , \varphi_{i} )$ for $i=1,\ldots , N$ denote the upper right-most vertex in the subdivision of  $[0,r]\times [0,r]$, and $\xi_{i}$ for $i=1,\ldots , N$ be the right-most endpoint in the subdivision of $[0,r]$, and set $p_{i} := (x_{i}, t_{i})$, where $x_{i} := F (\theta_{i}, \varphi_{i})$, $t_{i} := G(\xi_{i})$. Note that $p_{i} \in B_{\frac{r}{\sqrt{2}} } (x) \times  I_{\frac{r}{\sqrt{2}}} (t) \subset  B_{r} (x,t)$, and by construction we have that
    \begin{align*}
         d_{\mathbb{R}^{2}} \left( (\theta_{i}, \varphi_{i}) , (\theta_{j}, \varphi_{j})  \right) \geqslant c \, \frac{r}{\bigg\lfloor \frac{r}{\varepsilon^{\frac{1}{2p} }} \bigg \rfloor +1} \quad \text{ for all } i,j=1,\ldots , N
    \end{align*}
    for some constant $c>0$.  Set
    \begin{align*}
        A_{j} := \left\{ \max_{i=1, \ldots , j} \vert T(p_{i}) - T(x,t) \vert < \varepsilon \right\},
    \end{align*}
    and write 
    \begin{align*}
        \mathbb{P} (A_{j} ) = \mathbb{E} \left[ 1_{A_{j-1}} \, \mathbb{P} \left( \vert T(p_{i}) - T(x,t) \vert < \varepsilon \, \left. \right| T(p_{i}) \, : \, i=0, \ldots , j-1  \right) \right].
    \end{align*}
    By Proposition \ref{thm.non.local.deter} we have that 
    \begin{align*}
         &   \text{Var} \left( T(p_{j}) \, | \, T(p_{i}) \, : \, i=0, \ldots , j-1   \right) 
         \\
         & \geqslant c_{2} \min_{1\leqslant i \leqslant j-1}\rho^{2}_{\alpha} (d_{\mathbb{S}^{2}} (x_{j}, x_{i} )) =   c_{2} \min_{1\leqslant i \leqslant j-1} d_{\mathbb{S}^{2}} (x_{j}, x_{i} )^{\alpha-2}
         \\
         & \geqslant    c_{2} \min_{1\leqslant i \leqslant j-1}   d_{\mathbb{R}^{2}} \left( (\theta_{j}, \varphi_{j}) , (\theta_{i}, \varphi_{i})  \right)^{\alpha-2}
         \\
         & \geqslant c_{2} c^{\alpha -2}  \left(  \frac{r}{\bigg\lfloor \frac{r}{\varepsilon^{\frac{1}{2p}}} \bigg \rfloor +1} \right)^{\alpha-2} =: \gamma.
    \end{align*}
    By the estimate $x \leqslant \lfloor x \rfloor +1 \leqslant 2x$, for $x\geqslant 1$, it follows that 
    \begin{align}\label{eqn.a.bound}
        c_{2} \left( \frac{c}{2} \right)^{\alpha-2} \leqslant \frac{\gamma}{\varepsilon^{\frac{\alpha-2}{2p}}} \leqslant c_{2} c^{\alpha-2} .
    \end{align}
    The conditional distribution of $T(p_{j})$ given all the $T(p_{i})$ is Gaussian with conditional variance $ \text{Var} \left( T(p_{j}) \, | \, T(p_{i}) \, : \, i=0, \ldots , j-1   \right)$ since the field $T(p)$ is Gaussian. Then by the above lower bound on the conditional variance and Anderson's inequality \cite{Anderson}  it follows that 
    \begin{align*}
        & \mathbb{P} \left( \vert T(p_{j}) - T(x,t) \vert < \varepsilon \, \left. \right| T(p_{i}) \, : \, i=0, \ldots , j-1  \right) \leqslant \mathbb{P} \left( \vert Z \vert \leqslant \frac{\varepsilon^{\frac{\alpha-2}{2p}}}{\sqrt{\gamma}} \right)\leqslant \exp(-C),
    \end{align*}
    where $Z$ is a standard Gaussian random variable and the last inequality holds for some constant $C=C(\alpha, \beta)>0$ thanks to \eqref{eqn.a.bound}. Thus,
    \begin{align*}
        \mathbb{P} (A_{N}) \leqslant \exp(-C)  \mathbb{P} (A_{N-1}) \leqslant \cdots  \leqslant \exp( -NC),
    \end{align*}
    and hence 
    \begin{align*}
         & \mathbb{P} \left( M_{r} (x,t) < \varepsilon \right) \leqslant  \mathbb{P} \left( \max_{i=1, \ldots , N} \vert T(p_{i}) - T(x,t)) \vert < \varepsilon \right) = \mathbb{P} (A_{N}) 
         \\
         & \leqslant \exp( -NC) \leqslant \exp \left( - C \, \frac{r^{3}}{\varepsilon^{\frac{1}{p}}}  \right).
    \end{align*}
\end{proof}

\begin{prop}\label{thm.small.ball.estimate.lower.bound}
    Let $\{ T(x,t) , \, (x, t) \in \mathbb S^{2} \times [0,T] \}$ be a time-spherical Gaussian random field satisfying Assumption I. Then there exist a positive finite constant $A_{2} = A_{2} ( \alpha , \beta )$ such that for all $\varepsilon >0$ we have that 
    \begin{align}\label{eqn.small.ball.lower.bound}
    \mathbb{P} \left( M_{r} (x, t) < \varepsilon \right) \geqslant \exp \left( - A_{2} \psi (r, \varepsilon) \right)    ,
    \end{align}
    where $\psi= \psi_{\alpha ,\beta}$ is given by \eqref{eqn.psi}.
\end{prop}

\begin{proof}
    Let $N(B_{r} (x,t),d_{T} , \varepsilon)$ be the smallest number of $d_{T}$-balls of radius $\varepsilon$ needed to cover  $B_{r} (x,t)$. By Proposition \ref{prop.bound.dist} we know that the canonical metric $d_{T}$ is equivalent to $\mu_{\alpha, \beta}$, and hence it is enough to consider $N(B_{r} (x,t),\mu_{\alpha, \beta}  , \varepsilon)$. The $\mu_{\alpha, \beta}$-ball of radius $\varepsilon$ is given by
    \begin{align*}
        & B^{\mu_{\alpha, \beta}}_{\varepsilon}(x,t) = \left\{ (y,s) \in \mathbb{S}^{2} \times [0,T] \, : \,\mu_{\alpha, \beta} ((x,t), (y,s)) < \varepsilon \right\} 
        \\
        & = \left\{(y,s) \in \mathbb{S}^{2} \times [0,T] \, : \, |t-s|^{\beta} + (1- |t-s|^{\beta}) \rho_{\alpha}^{2} (d_{\mathbb S^{2}} (x,y) ) < \varepsilon^{2} \right\} .
    \end{align*}
    Due to isotropy and stationary of the field, we can assume  $t=0$ and $x=N$, the north pole. Then  $d(N,y) = \theta$ and in spherical coordinates we get 
    \begin{align*}
        & \text{Vol} ( B^{\mu_{\alpha, \beta}}_{\varepsilon}(N,0) ) = \int_{B^{\mu_{\alpha, \beta}}_{\varepsilon}(N,0)} dy ds = \int_{0}^{2\pi} d \varphi \int_{ \left\{ \vert s \vert^{\beta} \left( 1-\theta^{\alpha-2} \right) +  \theta^{\alpha-2} < \varepsilon^{2} \right\} } \sin \theta d\theta ds.
    \end{align*}
    Note that
    \begin{align*}
        &  \left\{ \vert s \vert^{\beta} \left( 1-\theta^{\alpha-2} \right) +  \theta^{\alpha-2} < \varepsilon^{2} \right\} 
        \\
        & =  \left\{ (s,\theta ) \, : \,   0 \leqslant \theta \leqslant \varepsilon^{\frac{2}{\alpha-2}} \, , \,   \vert s \vert < \left( \frac{\varepsilon^{2} - \theta^{\alpha-2}}{1-\theta^{\alpha-2}}\right)^{\frac{1}{\beta}} \right\} \cup \left\{ (s,\theta ) \, : \,  1 < \theta < \pi  \, , \,  \left( \frac{ \theta^{\alpha-2}- \varepsilon^{2} }{\theta^{\alpha-2}-1}\right)^{\frac{1}{\beta}} < s <T  \right\}  ,
    \end{align*}
    and since we are interested in the asymptotics near $(x,t) = (N,0)$, we can rule oiut the second set since $d(N, y) = \theta >1$. We then have that 
    \begin{align*}
        &  \text{Vol} ( B^{\mu_{\alpha, \beta}}_{\varepsilon}(N,0) ) = 2\pi \int_{0}^{\varepsilon^{\frac{2}{\alpha-2}}} \int_{- \left( \frac{\varepsilon^{2} - \theta^{\alpha-2}}{1-\theta^{\alpha-2}}\right)^{\frac{1}{\beta}}}^{\left( \frac{\varepsilon^{2} - \theta^{\alpha-2}}{1-\theta^{\alpha-2}}\right)^{\frac{1}{\beta}}} \sin \theta  ds d\theta 
        \\
        & = 2\pi \int_{0}^{\varepsilon^{\frac{2}{\alpha-2}}} 2 \left( \frac{\varepsilon^{2} - \theta^{\alpha-2}}{1-\theta^{\alpha-2}}\right)^{\frac{1}{\beta}}  \sin \theta d\theta =  4\pi \varepsilon^{\frac{2}{\beta} + \frac{2}{\alpha-2}} \int_{0}^{1} \left( \frac{1-u^{\alpha-2}}{1- \varepsilon^{2} u^{\alpha-2}} \right)^{\frac{1}{\beta}} \sin \left( \varepsilon^{\frac{2}{\alpha-2}} u \right) du
        \\
        & = 4\pi \varepsilon^{\frac{2}{\beta} + \frac{4}{\alpha-2}} \int_{0}^{1} u  \left( \frac{1-u^{\alpha-2}}{1- \varepsilon^{2} u^{\alpha-2}} \right)^{\frac{1}{\beta}} \frac{ \sin \left( \varepsilon^{\frac{2}{\alpha-2}} u \right) }{   \varepsilon^{\frac{2}{\alpha-2}} u } du =:  \varepsilon^{\frac{1}{p}  } F_{\alpha ,\beta} (\varepsilon),
    \end{align*}
    where the function $F_{\alpha ,\beta} (\varepsilon)$ is bounded uniformly in $\varepsilon$ above and below by a positive constant. Similarly, one can show that 
    \begin{align*}
        & \text{Vol}(B_{r} (N,0)) = 2 \pi r^{3} \int_{0}^{1} u (1-r u)^{\frac{1}{2}}   \frac{\sin (r u)}{ru } du =: r^{3} H(r),
    \end{align*}
    where the function $H(r)$ is bounded uniformly in $r$ above and below by a positive constant. Thus,
    \begin{align}
        N(B_{r} (N,0),\mu_{\alpha, \beta }  , \varepsilon)\leqslant \frac{\text{Vol} \left( B_{r} (N,0) \right)}{\text{Vol} (B^{\mu_{\alpha,\beta }}_{\varepsilon}(N, 0))} = \frac{r^{3} H(r)}{\varepsilon^{\frac{1}{p} } F_{\alpha ,\beta} (\varepsilon)},
    \end{align}
    and hence we can take 
    \begin{align*}
        \psi_{\alpha ,\beta} (r,\varepsilon ) : =  \frac{r^{3}}{\varepsilon^{\frac{1}{p} } },
    \end{align*}
    which satisfied the assumption of Lemma \ref{lemma.talagrand}.
\end{proof}

\begin{rmk}
    The proof of Proposition \ref{thm.small.ball.estimate.lower.bound} does not rely on the strong non-local determinism property from Proposition \ref{thm.non.local.deter}. In particular, it holds for any $\alpha \geqslant 4$ as well.
\end{rmk}

\section{Chung's law of the iterated logarithm}
The goal of this section is to show that the random variable 
\begin{align*}
    & \kappa_{\alpha , \beta}:= \liminf_{r\rightarrow 0} \phi (r) M_{r} (x,t),
\end{align*}
where 
\begin{align*}
    & M_{r} (x,t) := \max_{ (y,s)  \in B_{r} (x,t)} \vert T(y,s) - T(x,t) \vert,
\end{align*}
satisfies 
\begin{align*}
    \Prob \left( 0 < \kappa_{\alpha , \beta} < \infty \right) =1. 
\end{align*}
In the next section we will show that $\kappa_{\alpha , \beta}$ is constant a.s. as a byproduct of Theorem \ref{thm.small.deviations}. Let us recall that 
\begin{align*}
    & p = p(\alpha , \beta) =  \left( \frac{2}{\beta} + \frac{4}{\alpha -2} \right)^{-1}  & \phi (r) = \phi_{\alpha , \beta} (r) := \left( \frac{\log \vert \log r \vert}{r^{3}} \right)^{p}.
\end{align*}

Firs, let us prove that   $\kappa_{\alpha ,\beta}>0$ a.s. 

\begin{prop}\label{thm.lower.bound}
        Let $\{ T(x,t) , \, (x, t) \in \mathbb S^{2} \times [0,T] \}$ be a time-spherical Gaussian random field satisfying Assumption I. Then 
        \begin{equation}\label{eqn.chung.lower.bound}
            \kappa_{\alpha ,\beta}\geqslant A_{1}^{p}  \; \; \; a.s. 
        \end{equation}
        where $\phi (r)$, $p$, and $A_{1}$ are given by \eqref{eqn.phi}, \eqref{eqn.p.rate}, and \eqref{eqn.small.ball.upper.bound} respectively.
\end{prop}

\begin{proof}
    For $n \geqslant 1$ and $R>1$ let us set 
    \begin{align*}
        & r_{n} := R^{-n} , \quad  0 < \gamma < \left( \frac{  A_{1} }{R^{3}} \right)^{p} ,
        \\
        & A_{n} := \left\{ \phi (r_{n-1})M_{r_{n}} (x,t) \leqslant \gamma \right\}.
    \end{align*}
    Then by \eqref{eqn.small.ball.upper.bound} we have that 
    \begin{align*}
        & \mathbb{P} (A_{n} ) \leqslant \exp \left(  - A_{1} \psi \left( r_{n} , \frac{\gamma}{\phi (r_{n-1})} \right) \right) = \exp \left( - A_{1}   \frac{r_{n}^{3}}{\gamma^{\frac{1}{p}}}  \phi (r_{n-1})^{\frac{1}{p}}  \right),
        \\
        & = \exp \left( -A_{1} \frac{r_{n}^{3}}{ r_{n-1}^{3}} \log  | \log r_{n-1} |  \gamma^{- \frac{1}{p}} \right) = \exp \left(  -  \frac{A_{1}}{R^{3}} \gamma^{- \frac{1}{p}} \log ((n-1) \log R)\right)
        \\
        & = \left( \frac{1}{(n-1) \log R} \right)^{\frac{A_{1}}{R^{3}} \gamma^{- \frac{1}{p}} }
    \end{align*}
    and hence $\sum_{n=1}^{\infty} \mathbb{P} (A_{n} ) < \infty$ since $A_{1} > \gamma^{\frac{1}{p}} R^{3}$. Note that $\phi (s) > \phi (t)$ for $0<s<t \ll 1$ and then 
    \begin{align*}
        \inf_{r_{n} \leqslant r \leqslant r_{n-1}} \phi (r) \max_{(y,s)\in B_{r}(x,t)} \vert T(y,s) - T(x,t) \vert \geqslant \phi(r_{n-1}) \max_{ (y,s)\in B_{r_{n}}(x,t)} \vert T(y,s) - T(x,t) \vert .
    \end{align*}
Note that, for any $ 0 < \gamma <  \frac{ A_{1}^{p}  }{R^{3}}$
\begin{align*}
    & \mathbb{P} \left( \liminf_{r\rightarrow 0} \phi (r) \max_{ (y,s)\in B_{r}(x,t)} \vert T(y,s) - T(x,t) \vert < \gamma \right)
    \\
    & \leqslant \mathbb{P} \left( \bigcup_{k\geqslant 1} \bigcap_{n\geqslant k} \left\{ \inf_{r_{n} \leqslant r \leqslant r_{n-1}} \phi (r) \max_{(y,s)\in B_{r}(x,t)} \vert T(y,s) - T(x,t) \vert  < \gamma \right\}         \right)
    \\
    &  \leqslant \mathbb{P} \left( \bigcup_{k\geqslant 1} \bigcap_{n\geqslant k} \left\{ \phi(r_{n-1}) \max_{ (y,s)\in B_{r_{n}}(x,t)} \vert T(y,s) - T(x,t) \vert  < \gamma \right\}         \right) = \mathbb{P} \left( \liminf_{n\rightarrow \infty} A_{n} \right) =0,
\end{align*}
    where the latter is zero by the first Borel-Cantelli Lemma. Thus, we have that 
    \begin{align*}
        \liminf_{r\rightarrow 0} \phi (r) \max_{(y,s)\in B_{r}(x,t)} \vert T(y,s) - T(x,t) \vert  \geqslant \gamma \; \; a.s. \quad \text{ for every }  0 < \gamma < \left( \frac{ A_{1}  }{R^{3}} \right)^{p},
    \end{align*}
    and the result follows by letting $R$ go to 1 and  taking the supremum over $\gamma$ on both sides.
\end{proof}

We now prove that $\kappa_{\alpha , \beta}<\infty$ a.s.

\begin{prop}\label{thm.upper.bound}
        Let $\{ T(x,t) , \, (x, t) \in \mathbb S^{2} \times [0,T] \}$ be a time-spherical Gaussian random field satisfying Assumption I. Then 
        \begin{equation}\label{eqn.chung.lower.bound}
           \kappa_{\alpha ,\beta}  \leqslant A_{2, }^{p}  \; \; \; a.s. 
        \end{equation}
        where $\phi (r)$, $p$, and $A_{2}$ are given by \eqref{eqn.phi}, \eqref{eqn.p.rate}, and   \eqref{eqn.small.ball.lower.bound} respectively.
\end{prop}

\begin{proof}
    Let us set 
    \begin{align*}
        & r_{n} := e^{-n} \lfloor e^{n^{2}} \rfloor^{-1}, & d_{n} :=\lfloor e^{n^{2}} \rfloor ,
    \end{align*}
    and note that 
    \begin{align*}
        &r_{n} d_{n} = e^{-n} &  r_{n} d_{n+1} > e^{n}.
    \end{align*}
    Let us consider the random fields
    \begin{align*}
        & \Xi_{n} (x,t):= \sum_{\ell =d_{n}+1}^{d_{n+1}}\sum_{m= -\ell}^{\ell} a_{\ell m } (t) Y_{\ell m} (x) , &\Theta_{n} (x) := T(x,t) -  \Xi_{n} (x,t)
    \end{align*}
    that is, 
    \begin{align*}
        &\Theta_{n} (x,t) = \sum_{\ell = 1}^{d_{n}} \sum_{m= -\ell}^{\ell} a_{\ell m } (t) Y_{\ell m} (x)  + \sum_{\ell = d_{n+1} +1}^{\infty} \sum_{m= -\ell}^{\ell} a_{\ell m } (t) Y_{\ell m} (x).
    \end{align*}
    For a fixed $(x,t) \in \mathbb S^{2} \times [0,T]$, the sequence of random fields $\{ \Xi_{n} (x,t) , (x,t) \in \mathbb S^{2} \times [0,T] \}_{n=1}^{\infty}$ is independent.  Indeed, if $n<m$ then $d_{n} + 1< d_{n+1} < d_{m}$ and the sets
    \begin{align*}
        &\{ a_{\ell m} (t) , \, m = -\ell, \ldots , \ell , \; \; \ell = d_{n}+1  , \ldots , d_{n+1}  \} , & \{ a_{\ell m} (t) , \, m = -\ell, \ldots , \ell , \; \; \ell = d_{m}+1  , \ldots , d_{m+1}  \}
    \end{align*}
    are uncorrelated by \eqref{eqn.covariance.alm} and hence independent because they are Gaussian. Moreover, for each $n$ and $(x,t), \, (y,s) \in S^{2} \times [0,T]$, $\Xi_{n} (x,t)$ is independent of $\Theta_{n} (y,s)$ since they are Gaussian and 
    \begin{align*}
        & \E \left[ \Xi_{n} (x,t) \Theta_{n} (y,s) \right]=0,
    \end{align*}    
    by \eqref{eqn.covariance.alm}.    
    The proof is completed once we show that 
    \begin{align}
        & \sum_{n=1}^{\infty} \mathbb{P} \left( \phi (r_{n}) \max_{(y,s)\in B_{r_{n}} (x,t)} \vert  \Xi_{n} (y,s)- \Xi_{n} (x,t)  \vert \leqslant \gamma \right) = \infty , \quad \text{ for any } \gamma > A_{2}^{p},  \label{eqn.almost.there.1}
        \\
        & \sum_{n=1}^{\infty} \mathbb{P} \left( \phi (r_{n}) \max_{(y,s)\in B_{r_{n}} (x,t)} \vert \Theta_{n} (y,s)  - \Theta_{n} (x,t)\vert > \varepsilon \right) < \infty , \quad \text{ for any } \varepsilon >0. \label{eqn.almost.there.2}
    \end{align}
    Indeed, by Borel-Cantelli Lemmas and \eqref{eqn.almost.there.1} \eqref{eqn.almost.there.2} it follows that 
    \begin{align*}
       &  \liminf_{r\rightarrow 0} \phi (r) M_{r} (x,t) \leqslant \liminf_{n\rightarrow \infty } \phi (r_{n}) M_{r_{n}} (x,t)  \leqslant  \liminf_{n\rightarrow \infty } \phi (r_{n})  \max_{(y,s)\in B_{r_{n}} (x,t)} \vert \Xi_{n} (y,s) -  \Xi_{n} (x,t)  \vert 
       \\
       & + \liminf_{n\rightarrow \infty } \phi (r_{n})  \max_{(y,s)\in B_{r_{n}} (x,t)} \vert \Theta_{n} (y,s) - \Theta_{n} (x,t) \vert \leqslant \gamma + \varepsilon,
    \end{align*}
    for any $\gamma > A_{2}^{p}$ and $\varepsilon >0$. The proof is then completed by letting $\varepsilon$ go to zero and $\gamma$ to $A_{2}^{p}$. 
    
    Let us first prove \eqref{eqn.almost.there.1}. By Anderson inequality \cite{Anderson} and Proposition \ref{thm.small.ball.estimate.lower.bound} it follows that 
    \begin{align*}
         & \mathbb{P} \left( \phi (r_{n}) \max_{(y,s)\in B_{r_{n}} (x,t)} \vert \Xi_{n} (y,s) -  \Xi_{n} (x,t) \vert \leqslant \gamma \right)   
         \\
         & \geqslant   \mathbb{P} \left( \phi (r_{n}) \max_{(y,s)\in B_{r_{n}} (x,t)} \vert \Xi_{n} (y,s) -  \Xi_{n} (x,t) +  \Theta_{n}(y,s) - \Theta_{n} (x,t)  \vert \leqslant \gamma \right) 
         \\
         & =  \mathbb{P} \left( \phi (r_{n}) M_{r_{n}} (x) \leqslant \gamma \right) \geqslant \exp \left( - A_{2} \psi \left( r_{n} , \frac{\gamma}{\phi(r_{n})} \right)\right) 
         \\
         & = \exp \left( -A_{2} \gamma^{- \frac{1}{p}} \log | \log r_{n} | \right)= \left( \frac{1}{\vert \log r_{n} \vert} \right)^{A_{2} \gamma^{- \frac{1}{p}}},
    \end{align*}
    and \eqref{eqn.almost.there.1} follows by  Borel-Cantelli Lemma for independent events since $\gamma > A_{2}^{p}$.
    We  now prove \eqref{eqn.almost.there.2} by means of \eqref{lemma.talagrand.diameter} with $K= B_{r} (x,t)$ and $X =\Theta_{n}$. 
    First, let us bound the diameter $D$ of $B_{r}(x,t)$ with respect to the canonical metric  $d_{\Theta_{n}}$ induced by the field $\Theta_{n}$.
    We have that 
    \begin{align*}
        &\mathbb{E} \left[ ( \Theta_{n} (y,s) - \Theta_{n} (x,t) )^{2} \right] 
        \\
        & =  \sum_{\ell = 1}^{d_{n}}  \frac{2\ell +1}{2\pi} \left( C_{\ell} (0) - C_{\ell} (t-s)  P_{\ell} (\cos \theta) \right)  +  \sum_{\ell = d_{n+1}+1}^{\infty} \frac{2\ell +1}{2\pi} \left( C_{\ell} (0) - C_{\ell} (t-s)  P_{\ell} (\cos \theta) \right)  
        \\
        & = \vert t-s \vert^{\beta} \sum_{\ell = 1}^{d_{n}}  \frac{2\ell +1}{2\pi} C_{\ell} (0) + (1-  \vert t-s \vert^{\beta})  \sum_{\ell = 1}^{d_{n}}  \frac{2\ell +1}{2\pi} C_{\ell} (0) \left( 1-  P_{\ell} (\cos \theta)\right) 
        \\
        & + \vert t-s \vert^{\beta} \sum_{\ell = d_{n+1}+1}^{\infty}  \frac{2\ell +1}{2\pi} C_{\ell} (0) + (1-  \vert t-s \vert^{\beta})  \sum_{\ell = d_{n+1}+1}^{\infty}  \frac{2\ell +1}{2\pi} C_{\ell} (0) \left( 1-  P_{\ell} (\cos \theta)\right) ,
    \end{align*}
    where $\theta:= d_{\mathbb S^{2}} (x,y)$. Let us set 
    \begin{align*}
        & S_{1} :=  \max_{(y,s) \in B_{r} (x,t)} \vert t-s \vert^{\beta} \sum_{\ell = 1}^{d_{n}}  \frac{2\ell +1}{2\pi} C_{\ell} (0),
        \\
        &  S_{2} :=  \max_{(y,s) \in B_{r} (x,t)} (1-  \vert t-s \vert^{\beta})  \sum_{\ell = 1}^{d_{n}}  \frac{2\ell +1}{2\pi} C_{\ell} (0) \left( 1-  P_{\ell} (\cos \theta)\right) ,
        \\
        & S_{3} :=  \max_{(y,s) \in B_{r} (x,t)} \vert t-s \vert^{\beta} \sum_{\ell = d_{n+1}+1}^{\infty}  \frac{2\ell +1}{2\pi} C_{\ell} (0) ,
        \\
        & S_{4} :=  \max_{(y,s) \in B_{r} (x,t)}  (1-  \vert t-s \vert^{\beta})  \sum_{\ell = d_{n+1}+1}^{\infty}  \frac{2\ell +1}{2\pi} C_{\ell} (0) \left( 1-  P_{\ell} (\cos \theta)\right) ,
    \end{align*}
    so that 
    \begin{align*}
        \max_{(y,s) \in B_{r} (x,t)}\mathbb{E} \left[ ( \Theta_{n} (y,s) - \Theta_{n} (x,t) )^{2} \right] \leqslant  S_{1} + S_{2} + S_{3} + S_{4}. 
    \end{align*}
    Note that $\vert P_{\ell} (\cos \theta) \vert \leqslant 1$ and hence, by Assumption I and Lemma \ref{lemma.taylor.expansion.Pl} 
    \begin{align*}
        & S_{1} \leqslant c \, r_{n}^{\beta}  \sum_{\ell = 1}^{d_{n}}\ell^{1- \alpha} \leqslant   c \, r_{n}^{\beta}  \sum_{\ell = 1}^{d_{n +1}}\ell^{1- \alpha} \leqslant c \,   r_{n}^{\beta} d_{n +1}^{2-\alpha} = c\,   r_{n}^{\beta} \left( \frac{1}{d_{n +1}}\right)^{\alpha -2} \leqslant c \, r_{n}^{ \alpha -2 +\beta} e^{-n (\alpha -2)} ,
        \\
        & S_{2} \leqslant c \, \max_{|s-t| < r_{n}} (1-|t-s|^{\beta}) \sum_{\ell = 1}^{d_{n}}\ell^{1- \alpha} \max_{y\in B_{r} (x)} (1- P_{\ell} (\cos \theta) )   \leqslant c \,  \sum_{\ell = 1}^{d_{n}} \ell^{1-\alpha} ( \ell^{2} r_{n}^{2} + \ell^{4} r_{n}^{4} ) 
        \\
        &  = c \, r_{n}^{2}   \sum_{\ell = 1}^{d_{n}} \ell^{3-\alpha} +   c \, r_{n}^{4}   \sum_{\ell = 1}^{d_{n}} \ell^{5-\alpha}  \leqslant c\, ( r_{n}^{2} d_{n}^{4-\alpha} +  r_{n}^{4} d_{n}^{6-\alpha}  ) = c \,r_{n}^{\alpha-2} \left( (r_{n} d_{n} )^{4-\alpha} +(r_{n} d_{n} )^{6-\alpha} \right)
        \\
        & = c \, r_{n}^{\alpha-2} \left( e^{-n(4-\alpha)} + e^{-n(6-\alpha)} \right) \leqslant  c \, r_{n}^{\alpha-2} e^{-n(4-\alpha)},
        \\
        & S_{3} \leqslant c\, r_{n}^{\beta} \sum_{d_{n+1} +1}^{\infty} \ell^{1-\alpha} \leqslant c_{\alpha} r_{n}^{\beta} \left( \frac{1}{d_{n+1}}\right)^{\alpha-2} \leqslant c_{\alpha} r_{n}^{ \alpha -2 +\beta} e^{-n (\alpha -2)} ,
        \\
        & S_{4} \leqslant c \, \max_{|s-t| < r_{n}} (1-|t-s|^{\beta}) \sum_{d_{n+1} +1}^{\infty} \ell^{1-\alpha} \leqslant c_{\alpha} \, r_{n}^{\alpha -2} e^{-n (\alpha -2)},
    \end{align*}
    for some finite constant $c_{\alpha} >0$. Combining  everything together it follows that 
    \begin{align*}
        & D^{2} \leqslant S_{1} + S_{2} + S_{3} + S_{4} \leqslant c_{\alpha} r_{n}^{\alpha -2} e^{-n \lambda},
    \end{align*}
    where $\lambda = \lambda (\alpha ) := \min ( \alpha -2, 4 - \alpha )$. If $N_{\varepsilon}$ denotes the number of $d_{T}$-balls of radius $\varepsilon$ needed to cover $B_{r_{n}}(x,t)$, then proceeding as in the proof of Proposition \ref{thm.small.ball.estimate.lower.bound} we have that 
    \[
    N_{\varepsilon} \leqslant c \frac{r^{3}}{\varepsilon^{\frac{1}{p}}},
    \]
    for some finite constant $c>0$. Thus, 
    \begin{align*}
        & \int_{0}^{D} \sqrt{\log N_{\varepsilon}} d\varepsilon \leqslant \int_{0}^{c_{\alpha} r_{n}^{\frac{\alpha-2}{2}} e^{-\frac{n}{2}\lambda} } \sqrt{\log c  \frac{r_{n}^{3}}{\varepsilon^{\frac{1}{p}}} } d\varepsilon
        \\
        & = c_{\alpha, \beta} r_{n}^{3p}\int_{0}^{a} w^{-q} \sqrt{-\log w} dw,
    \end{align*}
    where $a:= d_{\alpha ,\beta} e^{-\frac{n}{2p} \lambda }  r_{n}^{\frac{\alpha-2}{2p} -3}$, for some constant $d_{\alpha ,\beta}$, and $q= 1-p<1$. In order to apply  Lemma \ref{lemma.integral.bound} we need $a<1$, which is satisfied if  $\frac{\alpha-2}{2p} -3 >0$, that is, if $\alpha > 2 + \beta$. By Lemma \ref{lemma.integral.bound}  it then follows that 
    \begin{align}
        &    \int_{0}^{D} \sqrt{\log N_{\varepsilon}} d\varepsilon \leqslant  c_{\alpha ,\beta}r_{n}^{3p} a^{1-q} \sqrt{- \log a} \notag
        \\
        & = c_{\alpha ,\beta} d_{\alpha \beta}^{p} r_{n}^{3p}  e^{-\frac{n}{2} \lambda }  r_{n}^{\frac{\alpha-2}{2} -3p} \sqrt{\frac{n}{2}\lambda \frac{1}{p} - \log d_{\alpha , \beta} + ( n+ n^{2} ) \left( \frac{\alpha-2}{2p}-3 \right) }  \notag
        \\
        & \leqslant c_{\alpha \beta}  r_{n}^{\frac{\alpha-2}{2}}   e^{-\frac{n}{2} \lambda } n. \label{eqn.constant}
    \end{align}
    Set 
    \begin{align*}
        &     u:= \frac{\varepsilon}{c \phi(r_{n}) } - c_{\alpha \beta}  r_{n}^{\frac{\alpha-2}{2}}   e^{-\frac{n}{2} \lambda } n
    \end{align*}
    where $c$ and $c_{\alpha \beta}$ are the constants  given in  \eqref{lemma.talagrand.diameter} and  \eqref{eqn.constant} respectively. Note that  $u>0$ for all $n$ large enough since  
    \begin{align*}
         & u= \frac{\varepsilon}{c} \frac{r_{n}^{3p}}{( \log \vert \log r_{n} \vert  )^{p}} \left( 1 - \frac{c}{\varepsilon} \left( r_{n}^{\frac{\alpha-2}{\beta} -1} \log \vert \log r_{n} \vert \right)^{p} e^{-\frac{n}{2}\lambda} n\right) >0.
    \end{align*}
     Then by   \eqref{lemma.talagrand.diameter}
    \begin{align*}
        & \mathbb{P} \left( \phi (r_{n}) \max_{(y,s)\in B_{r_{n}} (x,t)} \vert \Theta_{n} (y,s) - \Theta_{n} (x,t)   \vert > \varepsilon \right)
        \\
        &\leqslant \mathbb{P} \left(  \max_{(z, v), (y,s)\in B_{r_{n}} (x,t)} \vert  \Theta_{n} (y,s) - \Theta_{n} (z,v)  \vert > \frac{\varepsilon}{\phi (r_{n})} \right)
        \\
        & = \mathbb{P} \left(  \max_{(z, v), (y,s)\in B_{r_{n}} (x,t)} \vert \Theta_{n} (y,s) - \Theta_{n} (z,v) \vert > c \left( u+  c_{\alpha \beta}  r_{n}^{\frac{\alpha-2}{2}}   e^{-\frac{n}{2} \lambda } n \right) \right)
        \\
        & \leqslant \mathbb{P} \left(   \max_{(z, v), (y,s)\in B_{r_{n}} (x,t)} \vert \Theta_{n} (y,s) - \Theta_{n} (z,v) \vert > c \left( u+  \int_{0}^{D} \sqrt{\log N_{\varepsilon}} d\varepsilon \right) \right)\leqslant \exp \left( - \frac{u^{2}}{D^{2}}\right).
    \end{align*}
    Note that for $n$ large 
    \[
    u =  \frac{1}{\phi(r_{n})}  \left( \frac{\varepsilon}{c  } -   \phi(r_{n}) c_{\alpha \beta}  r_{n}^{\frac{\alpha-2}{2}}   e^{-\frac{n}{2} \lambda } n \right) \geqslant \frac{b}{\phi(r_{n})},
    \]
    for some constant $ b = b (\alpha , \beta )$ since $\phi(r_{n}) c_{\alpha \beta}  r_{n}^{\frac{\alpha-2}{2}}   e^{-\frac{n}{2} \lambda } n \rightarrow 0$.  Thus 
    \begin{align*}
        & \mathbb{P} \left( \phi (r_{n}) \max_{(y,s)\in B_{r_{n}} (x,t)} \vert \Theta_{n} (y,s) - \Theta_{n} (x,t)   \vert > \varepsilon \right) \leqslant \exp \left(- \frac{ b^{2}}{\phi (r_{n})^{2}} \frac{1}{D^{2}}  \right)
        \\
        & \leqslant \left( - C \frac{e^{n\lambda}}{ (\log \vert \log r_{n} \vert)^{p} \, r_{n}^{\alpha-2-{6p}}  }\right),
    \end{align*}
    for some finite constant $C>0$, and the proof of \eqref{eqn.almost.there.2} is then complete. 
\end{proof}

\section{Small probabilities}
In this section we prove Theorem \ref{thm.small.deviations}, the small ball principle for the random field $\{T(x,t) , (x,t) \in \mathbb{S}^{2} \times [0,T] \}$, and we show that the random variable $\kappa_{\alpha ,\beta}$ is constant a.s. By Proposition \ref{thm.lower.bound} and Proposition \ref{thm.upper.bound} we know that, for any $(x,t) \in \mathbb S^{2} \times [0,T]$ 
\begin{align*}
    & \kappa_{\alpha ,\beta} = \liminf_{r\rightarrow 0} \phi_{\alpha , \beta}  (r) M_{r} (x,t) , \; \; \; a.s.
    \\
    &A_{1}^{p} \leqslant \kappa_{\alpha ,\beta}\leqslant  A_{2}^{p} \; \; \; a.s.
\end{align*}
where $ A_{1}$,  $ A_{2}$, and $p$ are given by  \eqref{eqn.small.ball.upper.bound},   \eqref{eqn.small.ball.lower.bound}, and \eqref{eqn.p.rate} respectively.

\begin{proof}[Proof of Theorem \ref{thm.small.deviations}]
    Let 
    \begin{align*}
        & c_{+} := \limsup_{\varepsilon \rightarrow 0} - \frac{1}{\psi(r,\varepsilon)} \log \mathbb{P} \left(M_{r} (x,t) < \varepsilon \right) ,
        \\
        & c_{-} := \liminf_{\varepsilon \rightarrow 0} - \frac{1}{\psi(r,\varepsilon)} \log \mathbb{P} \left( M_{r} (x,t) < \varepsilon \right) ,
    \end{align*}
    where $\psi = \psi_{\alpha , \beta}$ is given by \eqref{eqn.psi}. By Proposition \ref{thm.small.ball.estimate.lower.bound} and Proposition \ref{thm.small.ball.estimate.upper.bound} we know that 
    \[
    0<A_{1} \leqslant c_{-} \leqslant c_{+} \leqslant A_{2} < \infty .
    \]
    The strategy of the proof is to show that 
    \begin{align*}
        c_{+} \leqslant   \kappa_{\alpha ,\beta }^{\frac{1}{p}} \leqslant c_{-} \; \; \; a.s.
    \end{align*}
    and in particular
    \begin{align*}
        & \kappa_{\alpha ,\beta }^{\frac{1}{p}} = c_{-} = c_{+} = \lim_{\varepsilon \rightarrow 0} - \frac{1}{\psi(r,\varepsilon)} \log \mathbb{P} \left( M_{r} (x,t) < \varepsilon \right)
    \end{align*}
    is constant. We first prove that $c_{+} \leqslant   \kappa_{\alpha ,\beta }^{\frac{1}{p}}$ a.s. Let us fix $\delta \in (0,c_{+})$. Then there exists an $\varepsilon (\delta)$ such that 
    \begin{equation}\label{eqn.sd.upper.bound}
            \mathbb{P} \left(M_{r} (x,t) < \varepsilon \right)  \leqslant \exp \left( - \delta \psi (r,\varepsilon) \right),
    \end{equation}
    for any $\varepsilon \leqslant \varepsilon (\delta)$. Let $r_{n} := R^{-n}$ for some $R>1$ and let us choose $\gamma$ such that $0<R^{3} \gamma < \delta$. Then 
    \begin{align*}
        & \mathbb{P} \left(M_{r_{n}}  (x,t) < \frac{\gamma^{p}}{\phi(r_{n-1})} \right) = \mathbb{P} \left( \max_{(y,s)\in B_{r_{n}} (x,t)} \vert T(y,s) - T(x,t) \vert < \frac{\gamma^{p}}{\phi(r_{n-1})} \right) 
        \\
        & \leqslant \exp \left( -\delta \psi  \left( r_{n} ,  \frac{\gamma^{p}}{\phi(r_{n-1})} \right) \right) = \exp \left( -\frac{\delta}{R^{3} \gamma} \log ( (n-1) \log R ) \right) = \left( \frac{1}{(n-1) \log R} \right)^{\frac{\delta}{R^{3} \gamma}},
    \end{align*}
    and hence 
    \begin{align*}
        & \sum_{n=1}^{\infty} \mathbb{P} \left(M_{r_{n}}  (x,t) < \frac{\gamma^{p}}{\phi(r_{n-1})}  \right)  < \infty,
    \end{align*}
    since $0<R^{3} \gamma < \delta$. Thus by Borel-Cantelli Lemma we have that 
    \begin{align*}
        & \mathbb{P} \left( \bigcap_{k \geqslant 1} \bigcup_{n\geqslant k} \left\{ M_{r_{n}}  (x,t) < \frac{\gamma^{p}}{\phi(r_{n-1})}  \right\} \right) =0 , \text{ i.e.} 
        \\
        & \mathbb{P} \left( \bigcup_{k \geqslant 1} \bigcap_{n\geqslant k} \left\{ M_{r_{n}}  (x,t) \geqslant \frac{\gamma^{p}}{\phi(r_{n-1})}  \right\} \right) =1,
    \end{align*}
    and hence almost surely for all large $n$ 
    \[
    M_{r_{n}}  (x,t) \geqslant \frac{\gamma^{p}}{\phi(r_{n-1})}  . 
    \]
    Note that $\phi(r) > \phi (r_{n-1})$  for $r_{n} \leqslant r \leqslant r_{n-1}$,  and it then follows that  
    \begin{align*}
        & M_{r}  (x,t) \geqslant M_{r_{n}}  (x,t) \geqslant \frac{\gamma^{p}}{\phi(r_{n-1})} \geqslant \frac{\gamma^{p}}{\phi(r)},
    \end{align*}
    which yields  
    \begin{align*}
        \kappa_{\alpha , \beta} := \liminf_{r\rightarrow 0} \phi(r)  M_{r}  (x,t) \geqslant  \gamma^{p} \; a.s.
    \end{align*}
    for any $\gamma < \frac{\delta}{R^{3}} < \frac{c_{+}}{R^{3}}$, and hence $c_{+} \leqslant   \kappa_{\alpha ,\beta }^{\frac{1}{p}}$ a.s. by letting first $R$ go to 1, and then $\delta$ to $c_{+}$. 

    Let us now prove that $\kappa_{\alpha ,\beta }^{\frac{1}{p}}   \leqslant c_{-}$ a.s. Let $\delta>c_{-}$ be fixed. Then there exists an $\varepsilon(\delta)$ such that 
    \begin{align*}
            \mathbb{P} \left(M_{r} (x,t) < \varepsilon \right)  \geqslant \exp \left( -\delta \psi  (r,\varepsilon) \right),
    \end{align*}
    for any $\varepsilon \leqslant \varepsilon (\delta)$. Let $r_{n}$ , $\Xi_{n} (x,t) $, and $\Theta_{n} (x,t)$ be defined as in the proof of Proposition \ref{thm.upper.bound}. Then 
    \begin{align*}
        & \mathbb{P} \left( \phi (r_{n} ) \max_{(y,s)\in B_{r_{n}} (x,t)} \vert \Xi_{n} (y,s)- \Xi_{n} (x,t) \vert \leqslant \delta \, c_{-}^{p-1} \right) 
        \\
        & \geqslant \mathbb{P} \left( \max_{(y,s)\in B_{r_{n}} (x,t)} \vert \Xi_{n} (y,s)- \Xi_{n} (x,t) \vert \leqslant   \frac{\delta \, c_{-}^{p-1}}{\phi (r_{n} )} \right)
        \\ 
        & \geqslant \exp \left( -\delta \psi \left( r_{n} ,   \frac{\delta \, c_{-}^{p-1}}{\phi (r_{n} )} \right) \right) = \exp \left( - \left(\frac{c_{-}}{\delta} \right)^{\frac{1-p}{p}}  \log \vert \log r_{n} \vert  \right)  = \left( \frac{1}{ \vert \log r_{n} \vert } \right)^{a},
    \end{align*}
    where $a:= \left(\frac{c_{-}}{\delta} \right)^{\frac{1-p}{p}} <1$ since $\delta> c_{-}$. The events 
    \[
    \left\{  \phi (r_{n} ) \max_{(y,s)\in B_{r_{n}} (x,t)} \vert \Xi_{n} (y,s)- \Xi_{n} (x,t) \vert \leqslant \delta \, c_{-}^{p-1} \right\}
    \]
    are independent, and hence by Borel-Cantelli Lemma for independent events it follows that 
    \[
    \liminf_{n\rightarrow \infty}  \phi (r_{n} ) \max_{(y,s)\in B_{r_{n}} (x,t)} \vert \Xi_{n} (y,s)- \Xi_{n} (x,t) \vert \leqslant  \delta \,  c_{-}^{p-1}  \, a.s. \, 
    \]
    for any  $\delta> c_{-}$, and thus, 
    \[
    \liminf_{n\rightarrow \infty}  \phi (r_{n} ) \max_{(y,s)\in B_{r_{n}} (x,t)} \vert \Xi_{n} (y,s)- \Xi_{n} (x,t) \vert \leqslant  \delta \, c_{-}^{p-1}  \leqslant c_{-}^{p} \; a.s. 
    \]
    Note that by \eqref{eqn.almost.there.2} we have that 
    \begin{align*}
        & \kappa_{\alpha ,\beta} := \liminf_{r \rightarrow 0}   \phi (r )M_{r} (x,t) \leqslant  \liminf_{n\rightarrow \infty}   \phi (r_{n} )M_{r_{n}}(x,t)
        \\
        & \leqslant  \liminf_{n\rightarrow \infty } \phi (r_{n})  \max_{(y,s)\in B_{r_{n}} (x,t)} \vert \Xi_{n} (y,s) -  \Xi_{n} (x,t)  \vert +  \liminf_{n\rightarrow \infty } \phi(r_{n})  \max_{(y,s)\in B_{r_{n}} (x,t)} \vert \Theta_{n} (y,s) - \Theta_{n} (x,t) \vert
        \\
        & \leqslant  c_{-}^{p} + \varepsilon
    \end{align*}
    for any $\varepsilon>0$, and hence the result follows by letting $\varepsilon$ go to zero.
\end{proof}

\appendix
\section{Proof of Proposition \ref{thm.non.local.deter} }\label{sec:appendixA}
In this section we prove Proposition \ref{thm.non.local.deter}, i.e. the strong non-local determinism for this field. We work in spherical coordinates $(\theta,\varphi)$ and we take without loss of generality $x_0=(0,0)$ to be the North Pole and $x_j=(\theta_j,\varphi_j)$ so that $d_{\mathbb{S}^2}(x_0,x_j)=\theta_j$. Since $T(x,t)$ is a Gaussian random field, 
 it is known that
\begin{equation*} Var(T(x_0,t_0)|T(x_1,t_1),...T(x_n,t_n))=\inf \{ \mathbb{E}[(T(x_0,t_0)-\sum_{j=1}^{n} \gamma_j T(x_j,t_j))^2]:\gamma_1,\dots,\gamma_n \in \mathbb{R}\}.
\end{equation*}
We observe that
    \begin{eqnarray*}
        \mathbb{E}\left[ \left( T(x,t)-\sum_{j=1}^{n} \gamma_j T(x_j,t_j)\right)^2\right]=\Gamma(x,t,x,t)-2\sum_{j=1}^{n} \gamma_j\Gamma(x,t,x_j,t_j)+\sum_{j,i=1}^{n} \gamma_i\gamma_j \Gamma(x_j,t_j,x_i,t_i) .
    \end{eqnarray*}
Indeed,
\begin{align*}
      &  \mathbb{E}\left[ \left( T(x,t)-\sum_{j=1}^{n} \gamma_j T(x_j,t_j)\right)^2\right] =  \mathbb{E}\left[ \left( \sum_{\ell =1}^{\infty} \sum_{m=- \ell}^{\ell} a_{\ell m}(t)Y_{\ell m}(x)-\sum_{\ell =1}^{\infty} \sum_{m=- \ell}^{\ell}\sum_{j=1}^{n}  \gamma_j a_{\ell m}(t_j)Y_{\ell m}(x_j) \right)^2\right]
      \\
        &= \sum_{\ell, \ell^{\prime} =1}^{\infty} \sum_{m =- \ell}^{\ell}   \sum_{m^{\prime}= -\ell^{\prime}}^{\ell^{\prime}} \mathbb{E}[a_{\ell m}(t)a_{\ell'm'}(t)] Y_{\ell m}(x) Y_{\ell'm'}(x) 
        \cr
        & \qquad - 2\sum_{\ell, \ell^{\prime} =1}^{\infty} \sum_{m =- \ell}^{\ell}   \sum_{m^{\prime}= -\ell^{\prime}}^{\ell^{\prime}}  \sum_{j=1}^{n} \gamma_j  \mathbb{E}[a_{\ell m}(t)a_{\ell'm'}(t_j)] Y_{\ell m}(x)Y_{\ell'm'}(x_j)
        \cr 
        & \qquad +\sum_{\ell, \ell^{\prime} =1}^{\infty} \sum_{m =- \ell}^{\ell}   \sum_{m^{\prime}= -\ell^{\prime}}^{\ell^{\prime}}  \sum_{j,j^{\prime} =1}^{n} \gamma_j \gamma_{j'}   \mathbb{E}[a_{\ell m}(t_j)a_{\ell'm'}(t_{j'})] Y_{\ell m}(x_j)Y_{\ell'm'}(x_{j'})
            \end{align*}
    In view of (\ref{eqn.covariance.alm}) we obtain
            
    \begin{align*}
      &  \mathbb{E}\left[ \left( T(x,t)-\sum_{j=1}^{n} \gamma_j T(x_j,t_j)\right)^2\right]
      \cr &   =\sum_{\ell =1}^{\infty} \sum_{m =- \ell}^{\ell}   C_\ell(0) Y_{\ell m}(x) Y_{\ell m}(x) - 2\sum_{\ell =1}^{\infty} \sum_{m =- \ell}^{\ell}  \sum_{j=1}^{n}   \gamma_j C_\ell(t-t_j) Y_{\ell m}(x)Y_{\ell'm'}(x_j)
        \\
        & \qquad + \sum_{\ell =1}^{\infty} \sum_{m =- \ell}^{\ell}  \sum_{ j,j^{\prime}=1}^{n} \gamma_j \gamma_{j'}  C_\ell(t_j-t_{j'}) Y_{\ell m}(x_j)Y_{\ell'm'}(x_{j'})
        \\
        &  =\sum_{\ell=1 }^{\infty} C_\ell(0) \frac{2\ell+1}{4\pi}P_\ell(0) - 2\sum_{\ell =1}^{\infty}\sum_{j=1}^{n}  \gamma_j C_\ell(t-t_j)  \frac{2\ell+1}{4\pi}P_\ell(d_{\mathbb{S}^2}(x,x_j))
        \\
        & \qquad+\sum_{\ell =1}^{\infty}\sum_{j, j^{\prime}=1}^{n}  \gamma_j \gamma_{j'}  C_\ell(t_j-t_{j'}) \frac{2\ell+1}{4\pi}P_\ell(d_{\mathbb{S}^2}(x_j,x_{j'})).
    \end{align*}
Then we write
  \begin{align*}
        &\mathbb{E}\left[ \left( T(x,t)-\sum_{j=1}^{n} \gamma_j T(x_j,t_j)\right)^2\right]=
       \sum_{\ell=0}^{\infty} \frac{2\ell+1}{4\pi} C_\ell(0) \bigg[ P_\ell(0)  -2\sum_{j=1}^{n} \gamma_j \frac{C_\ell(t-t_j)}{C_\ell(0)} P_\ell(d(x,x_j))
       \cr
       & \quad +  \sum_{j, j^{\prime}=1}^{n} \gamma_j\gamma_{j'}\frac{C_\ell(t_j-t_{j'})}{C_\ell(0)} P_\ell(d(x_{j'},x_j)) \bigg]
       \cr
       &=  \sum_{\ell=0}^{\infty} \frac{2\ell+1}{4\pi} C_\ell(0) 
       \bigg[ \left(1-\sum_{j=1}^{n} \gamma_j \frac{C_\ell(t-t_j)}{C_\ell(0)} P_\ell(d(x,x_j))\right)^2 \\& \quad  + 
       \sum_{j,j^{\prime} =1}^{n} \gamma_j\gamma_{j'} \bigg( \frac{C_\ell(t_j-t_{j'})}{C_\ell(0)} P_\ell(d(x_{j'},x_j))
       -\frac{C_\ell(t_j-t)C_\ell(t-t_{j'})}{C_\ell(0)C_\ell(0)}P_\ell(d(x,x_j))P_\ell(d(x,x_{j'}))\bigg) \bigg].
    \end{align*}

From Lemma \ref{lem:POSDEF} it follows that

\begin{eqnarray*}
        &&\mathbb{E}\left[ \left(T(x,t)-\sum_{j=1}^{n} \gamma_j T(x_j,t_j)\right)^2\right] \geq 
 \sum_{\ell=0}^{\infty} \frac{2\ell+1}{4\pi} C_\ell(0) 
 \left(1-\sum_{j=1}^{n} \gamma_j \frac{C_\ell(t-t_j)}{C_\ell(0)} P_\ell(d(x,x_j))\right)^2.
    \end{eqnarray*}

The thesis follows by proving the proposition below.

\begin{prop}
Assume condition I holds. For all $\varepsilon \in (0, \varepsilon_0/2]$ there exists a constant $C>0$ such that for all choices of $n\in \mathbb{N}$, all $(\theta_j, \varphi_j): \theta_j >\varepsilon$ and $\gamma_j \in \mathbb{R}$, $j=1,\dots,n$ we have 
\begin{eqnarray*}
&& \sum_{\ell=0}^{\infty} \frac{2\ell+1}{4\pi} C_\ell(0) 
 \left(1-\sum_{j=1}^{n} \gamma_j \frac{C_\ell(t-t_j)}{C_\ell(0)} P_\ell(d(x,x_j))\right)^2 \geq C \min_{1\leq k\leq n} \rho_{\alpha}^{2} (d_{\mathbb S^{2}} (x,x_k) )).
    \end{eqnarray*}
    \end{prop}

\begin{proof} 
For any fixed $\varepsilon>0$, let $\delta_\varepsilon(\cdot,\cdot)$ the spherical bump function constructed in \cite[Section 2.2]{MXL} with the corresponding coefficients $\{ b_{ \ell} (\varepsilon)\}$. We follow the proof of \cite[Proposition 7]{MXL}.
We define 
\[I:= \sum_{\ell=0}^{\infty} \sqrt{\frac{2\ell+1}{4\pi C_\ell(0)}}b_\ell(\varepsilon) \left[ \sqrt{\frac{2\ell+1}{4\pi}C_\ell(0)} \left( 1-\sum_{j=1}^{n}\gamma_j \frac{C_\ell(t-t_j)}{C_\ell(0)} P_\ell(d(x,x_j)) \right) \right].\]
By Cauchy-Schwarz 
\[I^2\leq \sum_{\ell=0}^{\infty} \frac{2\ell+1}{4\pi} \frac{b_\ell(\varepsilon)^2 }{C_\ell(0)}\sum_{\ell=0}^{\infty} \frac{2\ell+1}{4\pi} C_\ell(0) 
 \left(1-\sum_{j=1}^{n} \gamma_j \frac{C_\ell(t-t_j)}{C_\ell(0)} P_\ell(d(x,x_j))\right)^2  \]
 From which it follows that
 \[\sum_{\ell=0}^{\infty} \frac{2\ell+1}{4\pi} C_\ell(0) 
 \left(1-\sum_{j=1}^{n} \gamma_j \frac{C_\ell(t-t_j)}{C_\ell(0)} P_\ell(d(x,x_j))\right)^2 \geq \frac{I^2}{\sum_{\ell=0}^{\infty} \frac{b_\ell(\varepsilon)^2}{C_\ell(0)} \frac{2\ell+1}{4\pi}}.\]

 Let us compute I now. We have that
 \[
 I= \sum_{\ell-0}^{\infty} b_\ell(\varepsilon)\frac{2\ell+1}{4\pi} -\sum_{\ell=0}^{\infty} b_\ell(\varepsilon)\frac{2\ell+1}{4\pi} \sum_{j=1}^{n} \gamma_j \frac{C_\ell(t-t_j)}{C_\ell(0)} P_\ell(d(x,x_j)).
 \]
From \cite{MXL} we know that
 \[\sum_{\ell=0}^{\infty} b_\ell(\varepsilon)\frac{2\ell+1}{4\pi} = \delta_{\varepsilon}(0,0) \geq \frac{c'}{2\varepsilon^2}\]
 and $C_\ell(t-t_j)=(1-|t-t_j|^\beta)C_\ell(0)$ since  $|t-t_j|<1$, and hence 

\begin{eqnarray*}
   && \sum_{\ell=0}^{\infty} b_\ell(\varepsilon)\frac{2\ell+1}{4\pi} \sum_{j=1}^{n} \gamma_j \frac{C_\ell(t-t_j)}{C_\ell(0)} P_\ell(d(x,x_j))\\
    &&\quad=\sum_{j=1}^{n} \gamma_j (1-|t-t_j|^\beta)\sum_{\ell=0}^{\infty} b_\ell(\varepsilon)\frac{2\ell+1}{4\pi} P_\ell(d(x,x_j))\\&&
   \quad  = \sum_{j=1}^{n} \gamma_j (1-|t-t_j|^\beta) \delta_\varepsilon(\theta_j,\phi_j)=0.
\end{eqnarray*}
Then we have $I\geq \frac{c'}{2}\varepsilon^{-2}$. It remains to prove that 
\[\sum_{\ell=0}^{\infty} \frac{b_\ell(\varepsilon)^2}{C_\ell(0)} \frac{2\ell+1}{4\pi}=O(\varepsilon^{-(\alpha+2)}),\] which follows  from the proof of \cite[Proposition 7]{MXL} in view of Assumption I.
\end{proof}

\section{Technical Lemmas}
We collect here other technical lemmas exploited along the paper.

\begin{lem}\label{lem:POSDEF} Assume condition I. For any $(x_j, t_j)$ and $(x_{j'},t_{j'})$  points on $\mathbb{S}^2\times [0,T]$ and for any $\gamma_j \in \mathbb{R}$, $j=1,\dots, n$ we have that
\begin{align*}
\sum_{\ell=0}^{\infty} &\frac{2\ell+1}{4\pi} C_\ell(0)\\& \times
 \sum_{j,j^{\prime}=1}^{n} \gamma_j\gamma_{j'} \bigg( \frac{C_\ell(t_j-t_{j'})}{C_\ell(0)} P_\ell(d(x_{j'},x_j)) -\frac{C_\ell(t_j-t)C_\ell(t-t_{j'})}{C_\ell(0)C_\ell(0)}P_\ell(d(x,x_j))P_\ell(d(x,x_{j'})\bigg) \geq 0
\end{align*}
\end{lem}

\begin{proof}
   Note that
\begin{align*}
  & \frac{2\ell+1}{4\pi} C_\ell(0) 
 \bigg( \frac{C_\ell(t_j-t_{j'})}{C_\ell(0)} P_\ell(d(x_{j'},x_j)) -\frac{C_\ell(t_j-t)C_\ell(t-t_{j'})}{C_\ell(0)C_\ell(0)}P_\ell(d(x,x_j))P_\ell(d(x,x_{j'})\bigg)  
 \\
 & =\frac{4\pi}{(2\ell+1)C_\ell(0)} \bigg(C_\ell(0) \frac{2\ell+1}{4\pi} P_\ell(0) C_\ell(t_j-t_{j'})\frac{2\ell+1}{4\pi}P_\ell(d(x_{j'},x_j))
 \\
 & -C_\ell(t_j-t)P_\ell(d(x,x_j)) \frac{2\ell+1}{4\pi} C_\ell(t-t_{j'})\frac{2\ell+1}{4\pi}P_\ell(d(x,x_{j'})) \bigg)
\end{align*}
Let us set $C_{\ell}(x,y,t,s):=C_\ell(t-s)\frac{2\ell+1}{4\pi}P_\ell(d(x,y))$. Then  $C_{\ell}(x,y,t,s)$ is the covariance function of the random field $T_{\ell;t}(x):= \sum_{m=-\ell}^{\ell} a_{\ell m} (t) Y_{\ell m}(x)$. By Proposition 1 \cite{LuMaXi} the function $C_{\ell}(x,x,t,t)C_{\ell}(x_1,x_2,t_1,t_2)-C_{\ell}(x,x_1,t,t_1) C_{\ell}(x,x_2,t,t_2)$ is still a covariance function, and as such it is positive definite. Thus,
\begin{align*}
    & \sum_{\ell=0}^{\infty} \frac{2\ell+1}{4\pi} C_\ell(0)\\ & \quad \times 
 \sum_{j,j^{\prime}=1}^{n} \gamma_j\gamma_{j'} \bigg( \frac{C_\ell(t_j-t_{j'})}{C_\ell(0)} P_\ell(d(x_{j'},x_j)) -\frac{C_\ell(t_j-t)C_\ell(t-t_{j'})}{C_\ell(0)C_\ell(0)}P_\ell(d(x,x_j))P_\ell(d(x,x_{j'})\bigg)
 \\
 & =  \sum_{\ell=0}^{\infty} \frac{4\pi}{(2\ell+1)C_\ell(0)} \left( C_{\ell}(x,x,t,t,) C_{\ell}(x_{j} , x_{j^{\prime}}, t_{j} , t_{j^{\prime}}) - C_{\ell}(x, x_{j} , t ,t_{j}) C_{\ell}(x, x_{j^{\prime}} , t , t_{j^{\prime}})  \right) \geq 0
\end{align*}
since  $C_\ell(0)$ is also positive. 
\end{proof}

\begin{lem}\label{lemma.taylor.expansion.Pl}
    Let $P_{\ell}$ be the $\ell$-th Legendre polynomial. Then there exists a constant $c$ such that for any $\theta$ small enough 
    \begin{equation}
        0< 1- P_{\ell} (\cos \theta )\leqslant c (\ell^{2} \theta^{2} + \ell^{4} \theta^{4} ).
    \end{equation}
\end{lem}

\begin{proof}
    The Taylor expansion of $P_{\ell} (t)$ at $t=1$ reads
    \begin{align*}
        P_{\ell} (t) = 1 + P^{\prime}_{\ell} (t) (t-1)  + R_{\ell} (t)=1 + \frac{\ell (\ell +1)}{2} (t) (t-1)  + R_{\ell} (t).
    \end{align*}
    Note that 
    \begin{align*}
        & P^{\prime \prime}_{\ell} (t) \leqslant P^{\prime \prime}_{\ell} (1) = \frac{(\ell +2) !}{8 (\ell -2)!} \leqslant c \ell^{4},
    \end{align*}
    and 
    \begin{align*} 
       \vert  R_{\ell} (t) \vert = \frac{1}{2} \vert P^{\prime \prime}_{\ell} (\xi) \vert (t-1)^{2} \leqslant c \ell^{4} (t-1)^{2},
    \end{align*}
    for some $\xi \in [t,1]$. Thus,
    \begin{align*}
      &  1- P_{\ell} (\cos \theta) \leqslant  \frac{\ell (\ell +1)}{2} (1-\cos \theta ) +   c \ell^{4} (1-\cos \theta )^{2}
      \\
      & = \ell (\ell +1) \sin^{2} \frac{\theta}{2} + c \ell^{4} \sin^{4} \frac{\theta}{2} \leqslant c \left( \ell^{2} \theta^{2} + \ell^{4} \theta^{4} \right).
    \end{align*}
\end{proof}

\begin{lem}\label{lemma.integral.bound}
   For any $q<1$ and $0 < \delta <1$  there exists an explicit constant $c_{q, \delta}$ such that for any $0<a< \delta < 1$ 
   \begin{align}
       \int_{0}^{a} u^{-q} \sqrt{- \log u} du \leqslant c_{q, \delta} a^{1-q} \sqrt{-\log a}. \label{eqn.integral.bound}
   \end{align}
\end{lem}

\begin{proof}
    A simple integration by parts yields 
    \begin{align*}
       & \int_{0}^{a} u^{-q} \sqrt{- \log u} du = \frac{1}{1-q}  a^{1-q} \sqrt{-\log a} + \frac{1}{2(1-q)} \int_{0}^{a} u^{-q} ( - \log u )^{- \frac{1}{2}} du
        \\
        & \leqslant  \frac{1}{1-q}  a^{1-q} \sqrt{-\log a} + \frac{1}{2(1-q)}  ( - \log a )^{- \frac{1}{2}}  \int_{0}^{a} u^{-q}  du
        \\
        & = \sqrt{\vert \log a \vert} a^{1-q} \left( \frac{1}{1-q} + \frac{1}{2(1-q)^{2}}  \frac{1}{\vert \log a \vert} \right) \leqslant c_{q, \delta}  \sqrt{\vert \log a \vert} a^{1-q},
    \end{align*}
    where 
    \begin{align*}
        c_{q,\delta} :=  \frac{1}{1-q} + \frac{1}{2(1-q)^{2}}  \frac{1}{\vert \log \delta \vert} .
    \end{align*}
\end{proof}

\begin{acknowledgement}
    The authors want to thank Domenico Marinucci for helpful discussions during the preparation of this work.
    This article is dedicated to the memory of Orietta di Biagio.
\end{acknowledgement}

\end{document}